\documentclass[12pt]{amsart}
\usepackage[margin=1.2in]{geometry}
\usepackage{graphicx,latexsym}
\usepackage{tikz}
\usepackage{amsfonts, amssymb, amsmath, amsthm, bm, hyperref,verbatim}
\usepackage{tikz}
\usetikzlibrary{matrix,arrows}
\usepackage{mathrsfs}
\allowdisplaybreaks

\newcommand{\N}{\mathbb{N}}
\newcommand{\Z}{\mathbb{Z}}
\newcommand{\R}{\mathbb{R}}
\newcommand{\C}{\mathbb{C}}

\newcommand{\D}{\mathcal D}
\newcommand{\OO}{\mathcal{O}}

\newcommand{\OC}{\OO_{C}}

\newcommand{\csn}{\operatorname{csn}}

\newcommand{\ev}[2]{\langle#1,#2\rangle}

\newtheorem{theorem}{Theorem}[section]
\newtheorem{proposition}[theorem]{Proposition}
\newtheorem{lemma}[theorem]{Lemma}
\newtheorem{corollary}[theorem]{Corollary}

\theoremstyle{definition}
\newtheorem{definition}[theorem]{Definition}

\theoremstyle{remark}
\newtheorem{remark}[theorem]{Remark}

\numberwithin{equation}{section}

\newcommand\norm[1]{\left\lVert#1\right\rVert}

\DeclareMathOperator{\id}{id}

\DeclareMathOperator{\im}{Im}
\DeclareMathOperator{\supp}{supp}

\begin{document}
\title[Convolutors of translation-modulation invariant Banach spaces]{Convolutors of translation-modulation invariant Banach spaces of ultradistributions}

\author[L. Neyt]{Lenny Neyt}
\thanks{L. Neyt was supported by  FWO-Vlaanderen through the postdoctoral grant 12ZG921N}

\address{Department of Mathematics: Analysis, Logic and Discrete Mathematics\\ Ghent University\\ Krijgslaan 281\\ 9000 Gent\\ Belgium}
\email{lenny.neyt@UGent.be}


\begin{abstract}
We study the space of tempered ultradistributions whose convolutions with test functions are all contained in a given translation-modulation invariant Banach space of ultradistributions. Our main result will be the first structural theorem for the aforementioned space. As an application we consider several extensions of convolution. 
\end{abstract}

\maketitle

\section{Introduction}

Convolution constitutes as one of the most important tools in mathematical analysis. In the theory of generalized functions it has been an extensively studied subject going back to Schwartz's work \cite{S-ThDistr}, where new results are still found until this day \cite{B-O-ConvVVDist, D-P-V-NewDistSpTIB, O-ConvCondDist, O-W-AppWeighDLPSpConvDist, O-W-DistVFunc, W-ConvWeighDLPSp}. When looking at the theory of ultradistributions, a considerable amount of literature can be found on the existence of convolution in both the non-quasianalytic case \cite{D-P-P-V-ConvUltraDistTIB, K-K-P-EquivDefConvUltraDist, P-ConvBeurlingUltraDist, P-P-ConvRoumieuUltraDist} and the quasianalytic case \cite{D-P-V-ClassTISpQuasiAnalUltraDist, P-P-V-OnQuasiAnalClasGS}. The scope of this article is situated in the theory of the Gelfand-Shilov spaces \cite{G-S-GenFunc2} and their duals, i.e. the tempered ultradistributions. More specifically our main goal is to characterize when the convolution of a tempered ultradistribution with a test function is contained in a so-called translation-modulation invariant Banach space of ultradistributions (in short: TMIB) in the sense of \cite{D-P-P-V-TMIBUltraDist}. These spaces find their origin in harmonic analysis, offering a natural extension for several classes appearing there. Well-known examples of TMIB are the weighted $L^{p}$-spaces, modulation spaces \cite{F-ModSpLocCompactAbelianGroup, F-G-BanachSpIntGroupRepAtomicDecomp1, F-G-BanachSpIntGroupRepAtomicDecomp2, F-H-GaborFrameTFAnalDist} and their generalizations \cite{D-P-GaborFrameCharGenModSp, D-P-P-V-TMIBUltraDist}, and Wiener amalgam spaces \cite{F-BanachConvAlgWienerType, F-L-WienerAmalgamSpFundIdGaborAnal}.

We now give an overview of the content of this text. In the preliminary Sections \ref{sec:Preliminaries} and \ref{sec:WeighVVContFunc} we define and study weight sequences, the Pettis integral, and weighted vector-valued continuous functions. Then, in Section \ref{sec:GS}, we formally introduce the Gelfand-Shilov spaces of Beurling and Roumieu type, here commonly denoted by $\mathcal{S}^{[M]}_{[A]} = \mathcal{S}^{[M]}_{[A]}(\R^{d})$. In particular we consider convolution in their context and specifically show that the convolutions $* : \mathcal{S}^{[M]}_{[A]} \times \mathcal{S}^{[M]}_{[A]} \rightarrow \mathcal{S}^{[M]}_{[A]}$ and $* : \mathcal{S}^{\prime [M]}_{[A]} \times \mathcal{S}^{[M]}_{[A]} \rightarrow \mathcal{S}^{\prime [M]}_{[A]}$ are well-defined and (hypo-)continuous, but moreover may be seen as Pettis integrals, allowing us to employ a uniform definition of convolution for spaces in between $\mathcal{S}^{[M]}_{[A]}$ and $\mathcal{S}^{\prime [M]}_{[A]}$.

In Section \ref{sec:TMIB} we consider the TMIB. A Banach space $E$ is a TMIB if it is closed under translation and modulation, where the norms of these operations have certain bounds (see \eqref{eq:TMIBnormcond}), and such that the dense inclusions $\mathcal{S}^{[M]}_{[A]} \hookrightarrow E \hookrightarrow \mathcal{S}^{\prime [M]}_{[A]}$ hold. We will also consider the duals of TMIB (in short: DTMIB). Then to any TMIB or DTMIB $E$ we can associate the space $\D^{[M]}_{E}$ of those tempered ultradistributions whose derivatives are all elements of $E$ and such that the norms of those derivatives are bounded by the weight sequence $M$. Section \ref{sec:TMIB} mainly deals with the study of the structural and topological properties of $\D^{[M]}_{E}$.

The primary focus of this text are the convolutor spaces $\OC^{\prime}(\mathcal{S}^{[M]}_{[A]}, E)$ of those $f \in \mathcal{S}^{\prime [M]}_{[A]}$ such that $f * \varphi \in E$ for every $\varphi \in \mathcal{S}^{[M]}_{[A]}$ and $E$ is a given TMIB. In Section \ref{sec:ConvSp} we consider the topological and structural properties of $\OC^{\prime}(\mathcal{S}^{[M]}_{[A]}, E)$. In particular, our main goal is to prove the ensuing first structural theorem; all unexplained notions are defined in Sections \ref{sec:Preliminaries},  \ref{sec:GS}, and \ref{sec:TMIB}. 

	\begin{theorem}
		\label{t:StructThm}
		Let $M$ and $A$ be two weight sequences satisfying $(M.1)$ and $(M.2)'$, and let $E$ be a TMIB of class $[M] - [A]$ satisfying \eqref{eq:GrowthCondModulation}. Then, for a given $f \in \mathcal{S}^{\prime [M]}_{[A]}(\R^{d})$,
			\begin{equation}
				\label{eq:StructThmConv} 
				f * \varphi \in E , \qquad \forall \varphi \in \mathcal{S}^{[M]}_{[A]}(\R^{d}) , 
			\end{equation}
		if and only if there exist $f_{\alpha} \in E$, $\alpha \in \N^{d}$, such that
			\begin{equation}
				\label{eq:StructThmSum} 
				f = \sum_{\alpha \in \N^{d}} f^{(\alpha)}_{\alpha} , 
			\end{equation}
		and for some $\ell > 0$ (for every $\ell > 0$) we have that
			\begin{equation}
				\label{eq:StructThmBound} 
				\sup_{\alpha \in \N^{d}} \ell^{|\alpha|} M_{\alpha} \norm{f_{\alpha}}_{E} < \infty . 
			\end{equation}
	\end{theorem}

Note that due to the relatively weak conditions on the weight sequences one cannot use the parametrix method \cite{K-Ultradistributions1}, as it is commonly done in the literature. As a result, a novel approach is needed. In the special case where $E = C_{0}$, the space of continuous functions vanishing at infinity, Theorem \ref{t:StructThm} was already shown in \cite{D-N-V-SpUltraDistVanishInf} (as well as weighted variants) by Debrouwere, the author, and Vindas. In their proof an abstract criterion for the surjectivity of a continuous linear mapping in terms of its transpose (cfr. \cite[Lemma 1]{D-N-V-SpUltraDistVanishInf}, see also Lemma \ref{l:AbstractSurjCrit}) was combined together with the computation of the dual of $\dot{\mathcal{B}}^{\prime [M]}$, the space of ultradistributions vanishing at infinity. The latter was done by exploiting the mapping properties of the short-time Fourier transform (STFT) \cite{G-FoundationsTimeFreqAnal}. Here we extend the technique to its full potential in order to prove Theorem \ref{t:StructThm}. In particular, we will show that $\D^{[M]}_{E^{\prime}}$ can be seen as the dual of $\OC^{\prime}(\mathcal{S}^{[M]}_{[A]}, E)$ by use of the so-called desingularization formula of the STFT. It is interesting to note that the STFT has recently shown itself to be a powerful tool in studying the structure and topological invariants of (generalized) function spaces, and it will take up a central role in this text as well; for other examples we refer the reader to \cite{B-O-CharSConvMulSpSTFT, D-N-WeightPLBSpUltraDiffFuncMultSp, D-N-V-SpUltraDistVanishInf, D-V-WeightedIndUltra, D-V-TopPropConvSpSTFT, G-Z-SpacesTestFuncSTFT}.

In Section \ref{sec:DTMIB} we consider the case of a DTMIB $E^{\prime}$. For instance, we show that $\OC^{\prime}(\mathcal{S}^{[M]}_{[A]}, E^{\prime})$ is the strong dual of $\D^{[M]}_{E}$. This allows one, by more classical means via the Hahn-Banach theorem, to establish an analog of Theorem \ref{t:StructThm} in the case of a DTMIB in the form of Theorem \ref{t:StructThmDualDE}. Moreover, as $\mathcal{S}^{[M]}_{[A]}$ is not necessarily dense in $E^{\prime}$, we may also consider its closure in $E^{\prime}$, which we denote by $\dot{E}^{\prime}$. Then $\dot{E}^{\prime}$ is also a TMIB. In Theorem \ref{t:closureGSindualDE} we show that $\OC^{\prime}(\mathcal{S}^{[M]}_{[A]}, \dot{E}^{\prime})$ is exactly the closure of $\mathcal{S}^{[M]}_{[A]}$ in $\OC^{\prime}(\mathcal{S}^{[M]}_{[A]}, E^{\prime})$. This generalizes the special case where $E = L^{1}$ and we find $E^{\prime} = L^{\infty}$ and $\dot{E}^{\prime} = C_{0}$, showing the equivalent definition of $\dot{\mathcal{B}}^{\prime [M]}$ as the convolutor space of $C_{0}$ and as the closure of $\mathcal{S}^{[M]}_{[A]}$ in $\mathcal{B}^{\prime [M]}$, the space of bounded ultradistributions.

Finally, in Section \ref{sec:ExtensionConv}, we use the structure established in Theorem \ref{t:StructThm} in order to define three extensions of convolution. In Theorem \ref{t:ExtensionConvSmoothCase} we show that the convolutions of elements of $\OC^{\prime}(\mathcal{S}^{[M]}_{[A]}, E)$ with test functions is actually contained in $\D^{[M]}_{E}$ and moreover the space of test functions may be enlarged. For example, we see that the convolution $* : \dot{\mathcal{B}}^{\prime [M]} \times \D^{[M]}_{L^{1}} \rightarrow \dot{\mathcal{B}}^{[M]}$ holds, with $\dot{\mathcal{B}}^{[M]} = \D^{[M]}_{C_{0}}$ the space of ultradifferentiable functions vanishing at infinity. Theorem \ref{t:ExtensionConvSmoothCaseTransp} then looks at the transpose of the previous result, defining the convolution between $\OC^{\prime}(\mathcal{S}^{[M]}_{[A]}, E^{\prime})$ and $\OC^{\prime}(\mathcal{S}^{[M]}_{[A]}, E)$. In the last extension, Theorem \ref{t:ExtendConvUltraDistCase}, we consider a direct definition of convolution between two types of convolutor spaces. A particular example is the convolution $* : \dot{\mathcal{B}}^{\prime [M]} \times \D^{\prime [M]}_{L^{1}} \rightarrow \dot{\mathcal{B}}^{\prime [M]}$, where $\D^{\prime [M]}_{L^{1}} = \OC^{\prime}(\mathcal{S}^{[M]}_{[A]}, L^{1})$ is the convolutor space of $L^{1}$.

\section{Preliminaries}
\label{sec:Preliminaries}

We review two preliminary subjects which will be used throughout this text. First, we consider weight sequences through which we define our spaces of ultradifferentiable functions and ultradistributions in the sense of Komatsu \cite{K-Ultradistributions1}. Second, we discuss vector-valued integration in the form of Pettis integrals specific to our context, which will allow us to uniformly extend operations such as convolution to several spaces. 

\subsection{Weight sequences}

A sequence $M = (M_{p})_{p \in \N}$ of positive real numbers is called a \emph{weight sequence} if $\lim_{p \rightarrow \infty} (M_{p} / M_{0})^{1/p} = \infty$. We will often assume one or more of the following conditions on a weight sequence $M$:
	\begin{description}
		\item[$(M.1)$] $M_{p}^{2} \leq M_{p - 1} M_{p + 1}$, $p \geq 1$;
		\item[$(M.2)'$] $M_{p + 1} \leq C_{0} H^{p} M_{p}$ for some $C_{0}, H \geq 1$;
		\item[$(M.2)$] $M_{p + q} \leq C_{0} H^{p + q} M_{p} M_{q}$ for some $C_{0}, H \geq 1$;
		\item[$(M.3)'$] $\sum_{p = 1}^{\infty} M_{p - 1} / M_{p} < \infty$.
	\end{description}
We refer the reader to \cite{K-Ultradistributions1} for the meaning of these conditions. If $M$ satisfies $(M.1)$, then we have
	\begin{equation}
		\label{eq:M1product}
		M_{p} M_{q} \leq M_{0} M_{p + q} , \qquad \forall p, q \in \N .
	\end{equation} 
Throughout this text, the constants $C_{0}$ and $H$ are always subject to the conditions $(M.2)'$ and $(M.2)$. For any two weight sequences $M$ and $N$ we write $M \subset N$ if there are $C, L > 0$ such that $M_{p} \leq C L^{p} N_{p}$ for all $p \in \N$. Also, for a multi-index $\alpha \in \N^{d}$, we simply write $M_{\alpha} = M_{|\alpha|}$. Well-known examples are the Gevrey sequences $M_{G, s} = (p!^{s})_{p \in \N}$, $s > 0$, which always satisfy $(M.1)$ and $(M.2)$, while $(M.3)'$ is valid if and only if $s > 1$.

Let $M$ be a weight sequence. For any compact $K \Subset \R^{d}$ and $\ell > 0$ we write $\D^{M, \ell}_{K}$ for the Banach space of all $\varphi \in C^{\infty}(\R^{d})$ such that $\supp \varphi \subseteq K$ and
	\[ \|\varphi\|_{\D^{M, \ell}_{K}} = \sup_{(\alpha, x) \in \N^{d} \times K} \frac{|\varphi^{(\alpha)}(x)|}{\ell^{|\alpha|} M_{\alpha}} < \infty . \]
Then, we define the \emph{space of compactly supported ultradifferentiable functions (of Beurling and Roumieu type)}
	\[ \D^{(M)} = \varinjlim_{K \Subset \R^{d}} \varprojlim_{\ell \rightarrow 0^{+}} \D^{M, \ell}_{K} , \qquad \D^{\{M\}} = \varinjlim_{K \Subset \R^{d}} \varinjlim_{\ell \rightarrow \infty} \D^{M, \ell}_{K} . \]
We will use $\D^{[M]}$ as a common notation for $\D^{(M)}$ and $\D^{\{M\}}$. Moreover, we will often first state assertions for $\D^{(M)}$ followed in parenthesis by the corresponding statement for $\D^{\{M\}}$. Similar conventions will be used in the sequel for other spaces and notations. The strong dual of $\D^{[M]}$, denoted by $\D^{\prime [M]}$, is called the \emph{space of ultradistributions}. Now, if $M$ satisfies $(M.1)$, then, $\D^{[M]}$ is non-trivial if and only if $M$ satisfies $(M.3)'$ by \cite[Theorem 4.2]{K-Ultradistributions1}. For this reason, we call a weight sequence $M$ \emph{non-quasianalytic} if it satisfies $(M.3)'$. 

The \emph{associated function} of a weight sequence $M$ is defined as
	\[ \omega_{M}(t) = \sup_{p \in \N} \log \frac{t^{p} M_{0}}{M_{p}} , \qquad t > 0 , \]
and $\omega_{M}(0) = 0$. We define $\omega_{M}$ on $\R^{d}$ as the radial function $\omega_{M}(x) = \omega_{M}(|x|)$, $x \in \R^{d}$. Then, $\omega_{M}$ is a non-negative continuous function on $\R^{d}$ which vanishes in a neigborhood of the origin.

We will always assume that a weight sequence $M$ satisfies $(M.1)$ and $(M.2)'$. Consequently, the associated function enjoys the following properties:
	\begin{itemize}
		\item[$\bullet$] For any $x, y \in \R^{d}$ we have \cite[Lemma 2.1.3]{C-K-P-BoundValConvUltaDist}
			\begin{equation}
				\label{eq:M1}
				\omega_{M}(x + y) \leq \omega_{M}(2x) + \omega_{M}(2y)  .
			\end{equation}
		\item[$\bullet$] For any $k > 0$ we have the following bound \cite[Proposition 3.4]{K-Ultradistributions1},
			\begin{equation}
				\label{eq:M2'}
				\omega_{M}(t) - \omega_{M}(kt) \leq - \frac{\log(t / C_{0}) \log k}{\log H} , \qquad t > 0 . 
			\end{equation}
		In particular, we have that, for any $q > 0$, $e^{\omega_{M}(q \cdot)} / e^{\omega_{M}(qH^{d + 1} \cdot)} \in L^{1}$, a fact we will often use throughout this text. 
	\end{itemize}
	
\subsection{Pettis integrals}

Let $E$ be a lcHs (= locally convex Hausdorff space) and denote by $\csn(E)$ the set of all continuous seminorms on $E$. A function $f : \R^{d} \rightarrow E$ is called \emph{weakly integrable} if $\phi \circ f \in L^{1}$ for any $\phi \in E^{\prime}$. If there exists an $e \in E$ such that
	\[ \ev{\phi}{e} = \int_{\R^{d}} \ev{\phi}{f(x)} dx , \qquad \forall \phi \in E^{\prime} , \]
then we call $e$ the \emph{Pettis integral} of $f$ and we write $e = \int_{\R^{d}} f d\mu$, where $\mu$ denotes the Lebesgue measure. Note that the Pettis integral is unique and linear if it exists. If $F$ is another lcHs and $A : E \rightarrow F$ is a continuous linear map, then $A(f)$ is also weakly integrable, and if $f$ has a Pettis integral then so does $A(f)$ with
	\begin{equation}
		\label{eq:LinearMapPettisIntegral}
		A\left(\int_{\R^{d}} f d\mu\right) = \int_{\R^{d}} A(f) d\mu .
	\end{equation}
In particular, if $E$ is continuously embedded in $F$, and $f$ is weakly integrable with Pettis integral in $E$, then it is also weakly integrable in $F$ with the same Pettis integral. 

For any function $f : \R^{d} \rightarrow E$, it suffices that $p \circ f \in L^{1}$ for any $p \in \csn(E)$ in order for $f$ to be weakly integrable. Moreover, if $E$ is complete and $f$ is continuous, this also guarantees the existence of the Pettis integral.

	\begin{lemma}
		\label{l:SuffCondExistencePettisIntegral}
		Let $E$ be a complete lcHs and $f : \R^{d} \rightarrow E$ be continuous. If $p \circ f \in L^{1}$ for any $p \in \csn(E)$, then the Pettis integral of $f$ exists. Moreover, we have
			\begin{equation}
				\label{eq:UppdarBoundSeminormPettisIntegral}
				p\left(\int_{\R^{d}} f d\mu\right) \leq \int_{\R^{d}} p(f) d\mu , \qquad \forall p \in \csn(E) . 
			\end{equation}
	\end{lemma}
	
	\begin{proof}
		We first note that should the Pettis integral of $f$ exist, then \eqref{eq:UppdarBoundSeminormPettisIntegral} follows immediately from the bipolar theorem and \eqref{eq:LinearMapPettisIntegral}. Let $\chi$ be a continuous function on $\R^{d}$ so that $0 \leq \chi \leq 1$, $\chi(x) = 1$ for $|x| \leq 1$, and $\chi(x) = 0$ for $|x| > 2$. Now, for any $R > 0$ put $\chi_{R} = \chi(\cdot / R)$ and write $f_{R} = \chi_{R} \cdot f$. Then $f_{R}$ is a continuous function $\R^{d} \rightarrow E$ with compact support, whence its Pettis integral exists (by \cite[Theorem 3.27]{R-FuncAnal} combined with the corollary after \cite[Proposition 7.11]{T-TVSDistKern}). By \eqref{eq:UppdarBoundSeminormPettisIntegral} we see that $p(\int_{\R^{d}} f_{S} - f_{R} d\mu) \leq \int_{|x| > R} p(f) d\mu$ for any $0 < R < S$ and $p \in \csn(E)$, from which it follows that $(\int_{\R^{d}} f_{R} d\mu)_{R > 0}$ is a Cauchy net and thus converges to some $e \in E$. As $|\ev{\phi}{f_{R}(x)}| \leq |\ev{\phi}{f(x)}|$ for any $x \in \R^{d}$, $R > 0$, and $\phi \in E^{\prime}$, it follows from \eqref{eq:LinearMapPettisIntegral} and the dominated convergence theorem that
			\[ \ev{\phi}{e} = \lim_{R \rightarrow \infty} \ev{\phi}{\int_{\R^{d}} f_{R} d\mu} = \lim_{R \rightarrow \infty} \int_{\R^{d}} \ev{\phi}{f_{R}} d\mu = \int_{\R^{d}} \ev{\phi}{f} d\mu , \]
		for all $\phi \in E^{\prime}$. We may thus conclude that $e$ is the Pettis integral of $f$.
	\end{proof}

As a corollary of Lemma \ref{l:SuffCondExistencePettisIntegral}, we may now state the dominated convergence theorem for Pettis integrals.

	\begin{lemma}
		\label{l:PettisIntegralsDominatedConvergence}
		Let $E$ be a complete lcHs and $f_{n} : \R^{d} \rightarrow E$ be a sequence of continuous weakly integrable functions that converge pointwisely almost everywhere to a continuous function $f : \R^{d} \rightarrow E$. Suppose that for every $p \in \csn(E)$ there is a function $g_{p} \in L^{1}$ such that $p(f_{n}) \leq g_{p}$ for every $n \in \N$ almost everywhere. Then, $f$ is weakly integrable with Pettis integral and
			\[ \int_{\R^{d}} f d\mu = \lim_{n \rightarrow \infty} \int_{\R^{d}} f_{n} d\mu . \]
	\end{lemma}
	
	\begin{proof}
		If $f_{n} \rightarrow f$ pointwise almost everywhere with respect to the Lebesgue measure, then specifically $p(f) \rightarrow p(f_{n})$ pointwise almost everywhere for any $p \in \csn(E)$. Hence, by applying the scalar dominated convergence theorem we see that $p(f) \in L^{1}$ for any $p \in \csn(E)$. Lemma \ref{l:SuffCondExistencePettisIntegral} then shows that $f$ is weakly integrable with Pettis integral. Additionally, as $p(f - f_{n}) \leq 2g_{p}$ almost everywhere for any $p \in \csn(E)$, we see by another application of the scalar dominated convergence theorem that $\int_{\R^{d}} p(f - f_{n}) d\mu \rightarrow 0$. Also note that every $f_{n}$ has a Pettis integral by Lemma \ref{l:SuffCondExistencePettisIntegral}. Hence, by \eqref{eq:UppdarBoundSeminormPettisIntegral}, we find that $\int_{\R^{d}} f d\mu = \lim_{n} \int_{\R^{d}} f_{n} d\mu$. 
	\end{proof}
	
\section{Weighted vector-valued continuous functions}
\label{sec:WeighVVContFunc}

We now introduce spaces of vector-valued continuous functions weighted with respect to the associated function of some weight sequence. These may be thought of as frequency spaces, which will characterize the spaces we are interested in via the STFT. As a result, many properties of the frequency spaces are reflected in the associated spaces.

 Let $E$ be a lcHs and let $w : \R^{d} \rightarrow (0, \infty)$ be a continuous function. We define $C_{w}(E)$ as the lcHs of all continuous functions $\Phi : \R^{d} \rightarrow E$ such that
	\[ \sup_{\xi \in \R^{d}} w(\xi) p(\Phi(\xi)) < \infty , \qquad \forall p \in \csn(E) . \]
With respect to a weight sequence $M$, we then define the spaces
	\[ C_{(M)}(E) = \varprojlim_{q \rightarrow \infty} C_{e^{\omega_{M}(q \cdot)}}(E), \quad C_{\{M\}}(E) = \varinjlim_{q \rightarrow 0^{+}} C_{e^{\omega_{M}(q \cdot)}}(E) , \]
and
	\[ C_{(M), d}(E) = \varinjlim_{q \rightarrow \infty} C_{e^{- \omega_{M}(q \cdot)}}(E), \quad C_{\{M\}, d}(E) = \varprojlim_{q \rightarrow 0^{+}} C_{e^{- \omega_{M}(q \cdot)}}(E) .  \]	
In the special case where $E = \R^{d}$, we will simply write $C_{[M]} = C_{[M]}(\R^{d})$ and $C_{[M], d} = C_{[M], d}(\R^{d})$.

Topologically, the spaces we consider have the ensuing structure.

	\begin{lemma}
		\label{l:TopCMSp}
		\begin{itemize}
			\item[$(i)$] If $E$ is a Fr\'{e}chet space, then, $C_{(M)}(E)$ and $C_{\{M\}, d}(E)$ are Fr\'{e}chet spaces. Moreover, if $E$ is a Banach space, then the spaces are quasinormable.
			\item[$(ii)$] If $E = \varinjlim_{n} E_{n}$ is a compactly regular $(LB)$-space, then, $C_{\{M\}}(E)$ and $C_{(M), d}(E)$ are complete and thus regular $(LB)$-spaces, and moreover we have
				\[ C_{\{M\}}(E) = \varinjlim_{n \rightarrow \infty} C_{e^{\omega_{M}(\cdot / n)}}(E_{n}) \quad \text{and} \quad C_{(M), d}(E) = \varinjlim_{n \rightarrow \infty} C_{e^{-\omega_{M}(n \cdot)}}(E_{n}) . \]
		\end{itemize}
	\end{lemma}
	
	\begin{proof}
		$(i)$ We show the statement for $C_{(M)}(E)$, the case of $C_{\{M\}, d}(E)$ is analogous. Clearly $C_{(M)}(E)$ is a Fr\'{e}chet space. Now suppose $(E, \|\cdot\|_{E})$ is a Banach space. To show that $C_{(M)}(E)$ is quasinormable, it suffices to verify that \cite[Lemma 26.14]{M-V-IntroFuncAnal}
			\begin{gather*} 
				\forall q > 0 ~\forall r > 0 ~ \forall \varepsilon \in (0,1] ~\exists R > 0 
				\quad \forall \Phi \in C_{(M)}(E) \text{ with } \norm{\Phi}_{C_{e^{\omega_{M}(2^{\log H} q \cdot)}}(E)} \leq 1 \\ 
			 	\exists \Phi_{\varepsilon} \in C_{(M)}(E) \text{ with } \norm{\Phi_{\varepsilon}}_{C_{e^{\omega_{M}(r \cdot)}}(E)} \leq R \text{ such that } \norm{\Phi - \Phi_{\varepsilon}}_{C_{e^{\omega_{M}(q \cdot)}}(E)} \leq \varepsilon.
			\end{gather*}
		Take any $q > 0$. For any $\varepsilon \in (0, 1]$ we set $\theta = - \log_{2} \varepsilon$ and choose any continuous function $\chi_{\varepsilon} : \R^{d} \rightarrow [0, 1]$ such that $\chi_{\varepsilon}(\xi) = 1$ for $|\xi| \leq C_{0} e^{\theta} / q$ and $\chi_{\varepsilon}(\xi) = 0$ for $|\xi| \geq C_{0} e^{\theta} / q + 1$. For any $\Phi \in C_{(M)}(E)$ with $\norm{\Phi}_{C_{e^{\omega_{M}(2^{\log H} q \cdot)}}(E)} \leq 1$ we put $\Phi_{\varepsilon}  = \chi_{\varepsilon} \cdot \Phi$. Then, for arbitrary $r > 0$,
			\[ \sup_{\xi \in \R^{d}} \norm{\Phi_{\varepsilon}(\xi)}_{E} e^{\omega_{M}(r \xi)} = \sup_{\xi \in \R^{d}} |\chi_{\varepsilon}(\xi)| \norm{\Phi(\xi)}_{E} e^{\omega_{M}(r \xi)} \leq \sup_{|\xi| \leq C_{0} e^{\theta} / q + 1} \frac{e^{\omega_{M}(r \xi)}}{e^{\omega_{M}(2^{\log H} q \xi)}} .  \]
		On the other hand, we have by \eqref{eq:M2'},
			\begin{multline*}
				\sup_{\xi \in \R^{d}} \norm{\Phi(\xi) - \Phi_{\varepsilon}(\xi)}_{E} e^{\omega_{M}(q \xi)} \leq \\ \sup_{|\xi| > C_{0} e^{\theta} / q} \norm{\Phi(\xi)}_{E} e^{\omega_{M}(q \xi)} \leq \sup_{|\xi| > C_{0} e^{\theta} / q} e^{\omega_{M}(q \xi) - \omega_{M}(2^{\log H} q \xi)} \leq \sup_{|\xi| > C_{0} e^{\theta} / q} (1 / 2)^{\log q |\xi| / C_{0}} \leq \varepsilon . 
			\end{multline*}
	
		$(ii)$ We again only consider the statement for $C_{\{M\}}(E)$, the case of $C_{(M), d}(E)$ is done analogously. $C_{\{M\}}(E)$ is complete by \cite[Corollary 3.3]{B-M-S-ProjDescrWeighIndLim} and the remark following it shows that
			\[ C_{\{M\}}(E) = \lim_{m \rightarrow \infty} \lim_{n \rightarrow \infty} C_{e^{\omega_{M}(\cdot / m)}}(E_{n}) . \]
		Then $\varinjlim_{n} C_{e^{\omega_{M}(\cdot / n)}}(E_{n})$ is continuously contained in $C_{\{M\}}(E)$ and even more the spaces coincide as sets. By De Wilde's open mapping theorem it now follows that the spaces coincide topologically as well. 
	\end{proof}

The duality relation between the spaces $C_{[M]}(E)$ and $C_{[M], d}(E)$ is explored in the next result. 

	\begin{lemma}
		\label{l:InducedBilinearMap}
		Let $E$ be a Fr\'{e}chet space (a compactly regular $(LB)$-space), $F$ be a compactly regular $(LB)$-space (a Fr\'{e}chet space), $G$ be a complete lcHs, and let $\{ \Gamma_{\xi} : \xi \in \R^{d} \}$ be a subset of $L(G, G)$ such that
			\begin{gather} 
				\forall p_{0} \in \csn(G) ~ \exists p_{1} \in \csn(G) ~ \forall q_{0} > 0 ~ \exists q_{1} > 0 ~ ( \forall q_{1} > 0 ~ \exists q_{0} > 0) ~ \exists C > 0 ~ \forall g \in G : \nonumber \\ p_{0}(\Gamma_{\xi} g) e^{\omega_{M}(q_{0} \xi)} \leq C p_{1}(g) e^{\omega_{M}(q_{1} \xi)} .  \label{eq:InducedBilinearMapCond}
			\end{gather}
		Consider a hypocontinuous bilinear map $A : E \times F \rightarrow G$. Then,
			\begin{equation}
				\label{eq:InducedBilinearMap} 
				C_{[M]}(E) \times C_{[M], d}(F) \rightarrow G : \quad (\Phi_{E}, \Phi_{F}) \mapsto \int_{\R^{d}} \Gamma_{\xi} A(\Phi_{E}(\xi), \Phi_{F}(-\xi)) d\xi , 
			\end{equation}
		is a well-defined hypocontinuous bilinear map, where the integral in \eqref{eq:InducedBilinearMap} is a Pettis integral.
	\end{lemma}
	
	\begin{proof}
		As both $C_{[M]}(E)$ and $C_{[M], d}(F)$ are barreled, it suffices to show that the mapping in \eqref{eq:InducedBilinearMap} is separately continuous. Fix any $p_{0} \in \csn(G)$ and let $p_{1} \in \csn(G)$ be as in \eqref{eq:InducedBilinearMapCond}. Take any $\Phi_{F} \in C_{[M], d}(F)$. Then, for some $q_{0} > 0$ (for all $q_{0} > 0$) we have $\sup_{\xi \in \R^{d}} e^{- \omega_{M}(q_{0} \xi)} p(\Phi_{F}(- \xi)) < \infty$ for any $p \in \csn(F)$. Consequently, $\{ \Phi_{F}(- \xi) e^{-\omega_{M}(q_{0} \xi)} : \xi \in \R^{d} \}$ is a bounded set in $F$, so that $\{ A(\cdot, \Phi_{F}(- \xi)) e^{-\omega_{M}(q_{0} \xi)} : \xi \in \R^{d} \}$ is an equicontinuous subset of $L(E, G)$. In particular, there is some $p_{2} \in \csn(E)$ such that for some $C > 0$,
			\[ p_{1}(A(x, \Phi_{F}(- \xi))) \leq C p_{2}(x) e^{\omega_{M}(q_{0} \xi)} , \qquad \forall x \in E, \xi \in \R^{d} . \]
		By \eqref{eq:InducedBilinearMapCond}, for any $q_{1} > 0$ there is a $q_{2} > 0$ (for any $q_{2} > 0$ there is a $q_{1} > 0$) such that for some $C' > 0$,
			\[ p_{0}(\Gamma_{\xi} g) e^{-\omega_{M}(q_{2} \xi)} \leq C' p_{1}(g) e^{-\omega_{M}(q_{1} \xi)} , \qquad \forall g \in G, \xi \in \R^{d} . \]
		Take any $\Phi_{E} \in C_{[M]}(E)$. Then $\Phi_{E} \in C_{e^{\omega_{M}(q_{2} \cdot)}}(E)$, where in the Roumieu case $q_{2} > 0$ is fixed with respect to $\Phi_{E}$. It follows that
			\begin{multline*}
				p_{0}(\Gamma_{\xi} A(\Phi_{E}(\xi), \Phi_{F}(-\xi))) = p_{0}(\Gamma_{\xi} A(\Phi_{E}(\xi), \Phi_{F}(-\xi))) \frac{e^{\omega_{M}(q_{2} \xi)}}{e^{\omega_{M}(q_{2} \xi)}} \\
				\leq C' p_{1}(A(\Phi_{E}(\xi), \Phi_{F}(-\xi))) \frac{e^{\omega_{M}(q_{2} \xi)}}{e^{\omega_{M}(q_{1} \xi)}} 
				\leq C C' \frac{e^{\omega_{M}(q_{0} \xi)}}{e^{\omega_{M}(q_{1} \xi)}} p_{2}(\Phi_{E}(\xi)) e^{\omega_{M}(q_{2} \xi)} .
			\end{multline*}
		In the Beurling case we choose $q_{1} = H^{d + 1} q_{0}$ and let $q_{2}$ be determined by \eqref{eq:InducedBilinearMapCond}, in the Roumieu case we let $q_{1}$ be determined by \eqref{eq:InducedBilinearMapCond} and choose $q_{0} = H^{-d -1} q_{1}$. In particular, in both cases we have that $e^{\omega_{M}(q_{0} \cdot)} / e^{\omega_{M}(q_{1} \cdot)} \in L^{1}$. Consequently, we get for some $C'' > 0$,
			\[ \int_{\R^{d}} p_{0}(\Gamma_{\xi} A(\Phi_{E}(\xi), \Phi_{F}(-\xi))) d\xi \leq C'' \sup_{\xi \in \R^{d}} p_{2}(\Phi_{E}(\xi)) e^{\omega_{M}(q_{2} \xi)} . \]
		By Lemma \ref{l:SuffCondExistencePettisIntegral}, this not only shows that the Pettis integral in \eqref{eq:InducedBilinearMap} exists, but also the continuity of the mapping in the first variable. Analogously, one shows the continuity in the second variable, from which the result follows.
	\end{proof}

We end this section with two density results.

	\begin{lemma}
		\label{l:DensityCMSp1}
		Let $E, F$ be Fr\'{e}chet spaces (compactly regular $(LB)$-spaces) such that $E$ is dense in $F$. Then, $C_{[M]}(E)$ is dense in $C_{[M]}(F)$.
	\end{lemma}
	
	\begin{proof}
		Clearly, for the Beurling case,
			\[ C_{(M)} \widehat{\otimes}_{\varepsilon} E = \varprojlim_{q \rightarrow \infty} C_{e^{\omega_{M}(q \cdot)}} \widehat{\otimes}_{\varepsilon} E = \varprojlim_{q \rightarrow \infty} C_{e^{\omega_{M}(q \cdot)}}(E) = C_{(M)}(E) , \]
		and analogously $C_{(M)}(F) = C_{(M)} \widehat{\otimes}_{\varepsilon} F$. Similarly, we have in the Roumieu case $C_{\{M\}}(E) = C_{\{M\}} \widehat{\otimes}_{\varepsilon} E$ and $C_{\{M\}}(F) = C_{\{M\}} \widehat{\otimes}_{\varepsilon} F$ by \cite[Corollary 3.3]{B-M-S-ProjDescrWeighIndLim}. As $C_{[M]} \widehat{\otimes}_{\varepsilon} E$ is dense in $C_{[M]} \widehat{\otimes}_{\varepsilon} F$, our proof is complete.
	\end{proof}
	
	\begin{lemma}
		\label{l:DensityCMSp2}
		Let $E$ be a lcHs. Then, $C_{[M]}(E)$ is dense in $C_{[M], d}(E)$.
	\end{lemma}
	
	\begin{proof}
		In virtue of \eqref{eq:M2'}, one sees that the space of continuous functions $\R^{d} \rightarrow E$ with compact support is dense in both $C_{[M]}(E)$ and $C_{[M], d}(E)$. 
	\end{proof}
	
\section{Gelfand-Shilov spaces and their duals: definition, convolution, and the short-time Fourier transform}
\label{sec:GS}

We introduce the Gelfand-Shilov spaces and their duals, and consider convolution in this context. In particular we will show that the classical definition of convolution between Gelfand-Shilov spaces coincides with the vector-valued definition using Pettis integrals. The advantage of this is that in the sequel any definition of convolution using Pettis integrals is an automatic extension of that for Gelfand-Shilov spaces. As a special case we also consider the short-time Fourier transform. Throughout this section, we will work with two weight sequences $M$ and $A$ which both satisfy $(M.1)$ and $(M.2)'$.

\subsection{The Gelfand-Shilov spaces}

For any $q, \ell > 0$ we define $\mathcal{S}^{M, \ell}_{A, q} = \mathcal{S}^{M, \ell}_{A, q}(\R^{d})$ as the Banach space of all $\varphi \in C^{\infty}(\R^{d})$ such that
	\[ \norm{\varphi}_{\mathcal{S}^{M, \ell}_{A, q}} = \sup_{(\alpha, x) \in \N^{d} \times \R^{d}} \frac{|\varphi^{(\alpha)}(x)| e^{\omega_{A}(q x)}}{\ell^{|\alpha|} M_{\alpha}} < \infty . \]
Then, we define the \emph{Gelfand-Shilov spaces (of Beurling and Roumieu type)}
	\[ \mathcal{S}^{(M)}_{(A)} = \varprojlim_{\ell \rightarrow 0^{+}} \mathcal{S}^{M, \ell}_{A, \ell^{-1}} , \qquad \mathcal{S}^{\{M\}}_{\{A\}} = \varinjlim_{\ell \rightarrow \infty} \mathcal{S}^{M, \ell}_{A, \ell^{-1}} . \]
$\mathcal{S}^{(M)}_{(A)}$ is a $(FN)$-space, while $\mathcal{S}^{\{M\}}_{\{A\}}$ is a $(DFN)$-space \cite[Theorem 1.1]{{D-N-V-NuclGSKerThm}}. In particular, $\mathcal{S}^{\{M\}}_{\{A\}}$ is compactly regular.

Throughout this text, we will always assume that the space $\mathcal{S}^{(M)}_{(A)}$ is non-trivial. A sufficient condition for this is $M_{G, s} \subset M$ and $M_{G, r} \subset A$ for some $r, s > 0$ with $r + s > 1$ \cite[p.~235]{G-S-GenFunc2}. Other non-triviality conditions can be found in \cite{D-V-NonTrivialAnalyFunc}. 

A useful property we will often apply is the equivalent definition of the Gelfand-Shilov spaces by use of $L^{1}$-norms. For any measurable function $\omega : \R^{d} \rightarrow (0, \infty)$ for which $\omega$ and $\omega^{-1}$ are locally bounded, we denote by $L^{1}_{\omega} = L^{1}_{\omega}(\R^{d})$ the Banach space of all functions $\varphi : \R^{d} \rightarrow \C$ such that $\norm{\varphi}_{L^{1}_{\omega}} = \int_{\R^{d}} |\varphi(x)| \omega(x) dx < \infty$. For arbitrary $q, \ell > 0$, we now define $\mathcal{S}^{M, \ell}_{A, 1, q} = \mathcal{S}^{M, \ell}_{A, 1, q}(\R^{d})$ as the Banach space of all $\varphi \in C^{\infty}(\R^{d})$ such that
	\[ \norm{\varphi}_{\mathcal{S}^{M, \ell}_{A, 1, q}} = \sup_{\alpha \in \N^{d}} \frac{\norm{\varphi^{(\alpha)}}_{L^{1}_{e^{\omega_{A}(q \cdot)}}}}{\ell^{|\alpha|} M_{\alpha}} < \infty . \]
Then, we have as locally convex spaces \cite[Theorem 1.2]{D-N-V-NuclGSKerThm},
	\begin{equation}
		\label{eq:GSasL1Sp} 
		\mathcal{S}^{(M)}_{(A)} = \varprojlim_{\ell \rightarrow 0^{+}} \mathcal{S}^{M, \ell}_{A, 1, \ell^{-1}} , \qquad \mathcal{S}^{\{M\}}_{\{A\}} = \varinjlim_{\ell \rightarrow \infty} \mathcal{S}^{M, \ell}_{A, 1, \ell^{-1}} . 
	\end{equation}
	
We denote the translation and modulation operators by $T_{x} f(t) = f(t - x)$ and $M_{\xi} f(t) = e^{2 \pi i \xi \cdot t} f(t)$ for $x, \xi \in \R^{d}$. Also, we write $\check{f}(t) = f(-t)$ for reflection about the origin. Note that the space $\mathcal{S}^{[M]}_{[A]}$, and by transposition its dual space $\mathcal{S}^{\prime [M]}_{[A]}$, is closed under these operations.

The convolution between two functions $f \in L^{\infty}$ and $g \in L^{1}$ is defined as
	\[ f * g(t) = \int_{\R^{d}} g(x) f(t - x) dx . \]
In view of \eqref{eq:GSasL1Sp}, it follows that the convolution between two elements of $\mathcal{S}^{[M]}_{[A]}$ is well-defined. Even more, for arbitrary $q, \ell > 0$ and any $\varphi \in \mathcal{S}^{M, \ell}_{A, q}$ and $\psi \in L^{1}_{e^{\omega_{A}(q \cdot)}}$ we find, by \eqref{eq:M1}, that $\varphi * \psi \in \mathcal{S}^{M, \ell}_{A, q/2}$ with the ensuing upper bound for its norm,
	\begin{equation}
		\label{eq:ConvGSEstimate}
		\norm{\varphi * \psi}_{\mathcal{S}^{M, \ell}_{A, q / 2}} \leq \norm{\varphi}_{\mathcal{S}^{M, \ell}_{A, q}} \norm{\psi}_{L^{1}_{e^{\omega_{A}(q \cdot)}}} .
	\end{equation}
In particular, it follows that for any $\varphi, \psi \in \mathcal{S}^{[M]}_{[A]}$ also $\varphi * \psi \in \mathcal{S}^{[M]}_{[A]}$. Alternatively, the convolution between two elements of $\mathcal{S}^{[M]}_{[A]}$ may be expressed as a Pettis integral.
	
	\begin{lemma}
		\label{l:ConvolutionGS}
		For any two $\varphi, \psi \in \mathcal{S}^{[M]}_{[A]}$ we have
			\begin{equation}
				\label{eq:GSConvPettis} 
				\varphi * \psi = \int_{\R^{d}} \psi(x) T_{x} \varphi dx ,
			\end{equation}
		where the right-hand side is a Pettis integral. Moreover, the convolution
			\begin{equation} 
				\label{eq:GSConv}
				* : \mathcal{S}^{[M]}_{[A]} \times \mathcal{S}^{[M]}_{[A]} \rightarrow \mathcal{S}^{[M]}_{[A]} : \quad (\varphi, \psi) \mapsto \int_{\R^{d}} \psi(x) T_{x} \varphi dx , 
			\end{equation}
		is a continuous bilinear map.
	\end{lemma}
	
	\begin{proof}
		For the proof of \eqref{eq:GSConvPettis}, we follow a similar approach as that of \cite[Lemma 3.7]{D-P-V-ClassTISpQuasiAnalUltraDist}, however more direct and with a somewhat different construction so that only the condition $(M.2)'$ is needed (instead of $(M.2)$). For any $m, n \geq 1$ we denote $D_{m, n} = [-m, m]^{d} \cap (1 / n) \Z^{d}$ and for any $t \in D_{m, n}$ and $0 < \gamma < 1 / n$ we set $K_{m, n, \gamma, t} = ( t + [ - \frac{1}{2n} + \frac{\gamma}{2} , \frac{1}{2n} - \frac{\gamma}{2} ]^{d}) \cap [-m, m]^{d}$ . We then consider the function
			\[ L_{m, n, \gamma}(x) = \sum_{t \in D_{m, n}} \mu(K_{m, n, \gamma, t}) \psi(t) \varphi(x - t) ,  \]
		where $\mu(K_{m, n, \gamma, t})$ denotes the Lebesgue measure of $K_{m, n, \gamma, t}$. Clearly $L_{m, n, \gamma} \in \mathcal{S}^{[M]}_{[A]}$. Now, assume $\varphi, \psi \in \mathcal{S}^{M, \ell}_{A, q}$, then both $\varphi * \psi$ and $L_{m, n, \gamma}$ are elements of $\mathcal{S}^{M, \ell}_{A, q / 2}$. We show that there exists an adequate sequence $(m_{k}, n_{k}, \gamma_{k})_{k \in \N}$ such that $L_{m_{k}, n_{k}, \gamma_{k}} \rightarrow \varphi * \psi$ in $\mathcal{S}^{M, H \ell}_{A, q / (2 H^{d+1})}$. We set
			\[ K^{1}_{m, n, \gamma} = \bigcup_{t \in D_{m, n}} K_{m, n, \gamma, t} \quad \text{and} \quad K^{2}_{m, n, \gamma} = [-m, m]^{d} \setminus K^{1}_{m, n, \gamma} , \]
		then, for any $\alpha \in \N^{d}$,
			\begin{align*} 
				&|\partial^{\alpha} [\varphi * \psi(x) - L_{m, n, \gamma}(x)]| \\
				&\qquad \qquad \leq \int_{\R^{d} \setminus [-m, m]^{d}} |\psi(u)| |\varphi^{(\alpha)}(x - u)| du  \\ 		
				&\qquad \qquad \phantom{\leq} + \sum_{t \in D_{m, n}} \int_{K_{m, n, \gamma, t}} |\psi(u) \varphi^{(\alpha)}(x - u) - \psi(t) \varphi^{(\alpha)}(x - t)| du \\
				&\qquad \qquad \phantom{\leq} + \int_{K^{2}_{m, n, \gamma}} |\psi(u)| |\varphi^{(\alpha)}(x - u)| du  \\
				&\qquad \qquad =: S^{\alpha}_{m, n, \gamma, 1}(x) + S^{\alpha}_{m, n, \gamma, 2}(x) + S^{\alpha}_{m, n, \gamma, 3}(x) .
			\end{align*}
		Fix an arbitrary $\varepsilon > 0$. We will, step by step, provide upper bounds for the $S^{\alpha}_{m, n, \gamma, j}$'s in function of $\varepsilon$. First of all, we have
			\begin{align*} 
				S^{\alpha}_{m, n, \gamma, 1}(x) &= \int_{\R^{d} \setminus [-m, m]^{d}} |\psi(u)| |\varphi^{(\alpha)}(x - u)| du \\
				&\leq M_{0} \norm{\psi}_{\mathcal{S}^{M, \ell}_{A, q}} \norm{\varphi}_{\mathcal{S}^{M, \ell}_{A, q / H^{d+1}}} l^{|\alpha|} M_{\alpha} e^{- \omega_{A}(q x / (2 H^{d+1}))} \int_{\R^{d} \setminus [-m, m]^{d}} \frac{e^{\omega_{A}(q u / H^{d+1})}}{e^{\omega_{A}(q u)}} du .  
			\end{align*}
		As $e^{\omega_{A}(q \cdot / H^{d+1})} / e^{\omega_{A}(q \cdot)} \in L^{1}$, there exists a $m_{\varepsilon} \geq 1$ such that
			\[ \sup_{x \in \R^{d}} \frac{S^{\alpha}_{m_{\varepsilon}, n, \gamma, 1}(x) e^{\omega_{A}(q x / (2 H^{d+1}))}}{l^{|\alpha|} M_{\alpha}} \leq \frac{\varepsilon}{3} . \]
		Next, by Taylor expanding $\psi(u) \varphi^{(\alpha)}(x - u)$ at $t$, we get by \eqref{eq:M1product}
			\begin{align*}
				&S^{\alpha}_{m_{\varepsilon}, n, \gamma, 2}(x) = \sum_{t \in D_{m, n}} \int_{K_{m_{\varepsilon}, n, \gamma, t}} |\psi(u) \varphi^{(\alpha)}(x - u) - \psi(t) \varphi^{(\alpha)}(x - t)| du \\
				&\leq \sum_{t \in D_{m, n}} \int_{K_{m_{\varepsilon}, n, \gamma, t}} \sum_{|\beta| = 1} |(u - t)^{\beta}| \int_{0}^{1} |\psi^{(\beta)}(t + s(u - t)) \varphi^{(\alpha)}(x - t - s(u - t))| \\
				&\hspace{6cm} + |\psi(t + s(u - t)) \varphi^{(\alpha + \beta)}(x - t - s(u - t))| ds du \\
				&\leq \frac{1}{2n} \norm{\psi}_{\mathcal{S}^{M, \ell}_{A, q}} \norm{\varphi}_{\mathcal{S}^{M, \ell}_{A, q}} \sum_{|\beta| = 1} \ell^{|\alpha + \beta|} (M_{\beta} M_{\alpha} + M_{0} M_{\alpha + \beta}) \\
				&\hspace{4.5cm} \sum_{t \in D_{m, n}} \int_{K_{m_{\varepsilon}, n, \gamma, t}} \int_{0}^{1} e^{- \omega_{A}(q[t + s(u - t)]) - \omega_{A}(q[x - t - s(u - t)])} ds du \\
				&\leq C_{0} M_{0} \ell \norm{\psi}_{\mathcal{S}^{M, \ell}_{A, q}} \norm{\varphi}_{\mathcal{S}^{M, \ell}_{A, q}} \frac{d (2m_{\varepsilon})^{d}}{n} (H \ell)^{|\alpha|} M_{\alpha} e^{- \omega_{A}(q x / 2)} .
			\end{align*}
		Consequently, there exists a $n_{\varepsilon} \geq 1$ large enough so that
			\[ \sup_{x \in \R^{d}} \frac{S^{\alpha}_{m_{\varepsilon}, n_{\varepsilon}, \gamma, 2}(x) e^{\omega_{A}(q x / 2)}}{(H \ell)^{|\alpha|} M_{\alpha}} \leq \frac{\varepsilon}{3} . \]
		Finally, we have that
			\begin{align*} 
				S^{\alpha}_{m_{\varepsilon}, n_{\varepsilon}, \gamma, 3}(x) &= \int_{K^{2}_{m_{\varepsilon}, n_{\varepsilon}, \gamma}} |\psi(u)| |\varphi^{(\alpha)}(x - u)| du \\
				&\leq M_{0} \norm{\psi}_{\mathcal{S}^{M, \ell}_{A, q}} \norm{\varphi}_{\mathcal{S}^{M, \ell}_{A, q / H^{d+1}}} l^{|\alpha|} M_{\alpha} e^{- \omega_{A}(q x / (2 H^{d+1}))} \int_{K^{2}_{m_{\varepsilon}, n_{\varepsilon}, \gamma}} \frac{e^{\omega_{A}(q u / H^{d+1})}}{e^{\omega_{A}(q u)}} du .  
			\end{align*}
		Note that $\mu(K^{2}_{m_{\varepsilon}, n_{\varepsilon}, \gamma}) \rightarrow 0^{+}$ as $\gamma \rightarrow 0^{+}$. Hence, as $e^{\omega_{A}(q \cdot / H^{d+1})} / e^{\omega_{A}(q \cdot)} \in L^{1}$, there exists a $0 < \gamma_{\varepsilon} < 1 / n_{\varepsilon}$ small enough so that 
			\[ \sup_{x \in \R^{d}} \frac{S^{\alpha}_{m_{\varepsilon}, n_{\varepsilon}, \gamma_{\varepsilon}, 3}(x) e^{\omega_{A}(q x / (2 H^{d+1}))}}{l^{|\alpha|} M_{\alpha}} \leq \frac{\varepsilon}{3} . \]
		In conclusion, we find
			\[ \sup_{(\alpha, x) \in \N^{d} \times \R^{d}} \frac{|\partial^{\alpha} [\varphi * \psi(x) - L_{m_{\varepsilon}, n_{\varepsilon}, \gamma_{\varepsilon}}(x)]| e^{\omega_{A}(q x / (2 H^{d+1}))}}{(H \ell)^{|\alpha|} M_{\alpha}} \leq \varepsilon , \]
		so that $L_{m_{1/\kappa}, n_{1/\kappa}, \gamma_{1/\kappa}} \rightarrow \varphi * \psi$ in $\mathcal{S}^{M, H \ell}_{A, q / (2 H^{d+1})}$ as $\kappa \rightarrow \infty$. From here we may conclude that there exists a sequence $(m_{k}, n_{k}, \gamma_{k})_{k}$ such that
			\[ L_{m_{k}, n_{k}, \gamma_{k}} \rightarrow \varphi * \psi \qquad \text{as } k \rightarrow \infty \text{ in } \mathcal{S}^{[M]}_{[A]} . \]
		Indeed, in the Roumieu case this follows directly, while in the Beurling case this can be concluded by applying a diagonal argument. Moreover, we may assume that $\mu(K^{2}_{m_{k}, n_{k}, \gamma_{k}}) \leq 2^{- k}$ for any $k \geq 1$, a technicality we will need later. Denote by $\chi_{K_{m, n, \gamma, t}}$ the characteristic function of $K_{m, n, \gamma, t}$. We define
			\[ e_{m, n, \gamma}(x) = \sum_{t \in D_{m, n}} \chi_{K_{m, n, \gamma, t}}(x) \psi(t) T_{t} \varphi . \]
		Then, $e_{m, n, \gamma}$ is a weakly integrable function $\R^{d} \rightarrow \mathcal{S}^{[M]}_{[A]}$ with Pettis integral 
			\[ \int_{\R^{d}} e_{m, n, \gamma}(x) dx = L_{m, n, \gamma} . \] 
		Moreover, $e_{m_{k}, n_{k}, \gamma_{k}}(x)$ converges pointwisely almost everywhere to $\psi(x) T_{x} \varphi$. Indeed, let $N \subset \R^{d}$ be the set of all $x \in \R^{d}$ such that $x \in K^{2}_{m_{k}, n_{k}, \gamma_{k}}$ for infinitely many $k$. Note that for any $k_{0} \geq 1$ we have that $N \subseteq \bigcup_{k \geq k_{0}} K^{2}_{m_{k}, n_{k}, \gamma_{k}}$, whence $\mu(N) \leq \sum_{k \geq k_{0}} \mu(K^{2}_{m_{k}, n_{k}, \gamma_{k}}) \leq 2^{- k_{0} + 1}$, so that $N$ is a null set. On the other hand, for any $x \notin N$ it is clear that $e_{m, n, \gamma}(x) \rightarrow \psi(x) T_{x} \varphi$ in $\mathcal{S}^{[M]}_{[A]}$. Now,
			\[ \norm{e_{m_{k}, n_{k}, \gamma_{k}}(x)}_{\mathcal{S}^{M, \ell}_{A, q/ (2H^{d + 1})}} \leq \norm{\varphi}_{\mathcal{S}^{M, \ell}_{A, q}} \sup_{|\theta| \leq \sqrt{d}} |\psi(x + \theta)| e^{\omega_{A}(q (x + \theta) / H^{d + 1})} . \]
		For $|x| \geq 2 \sqrt{d}$, it follows from \eqref{eq:M2'} that
			\[ \sup_{|\theta| \leq \sqrt{d}} |\psi(x + \theta)| e^{\omega_{A}(q (x + \theta) / H^{d + 1})} \leq \left( \frac{2 H^{d + 1} C_{0}}{q} \right)^{d + 1} \norm{\psi}_{\mathcal{S}^{M, \ell}_{A, q}} x^{-(d + 1)} , \]
		so that in particular we get $\norm{e_{m_{k}, n_{k}, \gamma_{k}}(x)}_{\mathcal{S}^{M, \ell}_{A, q/2 H^{d + 1}}} \leq g$ for some $g \in L^{1}$ and for all $k \in \N$. Hence, by applying Lemma \ref{l:PettisIntegralsDominatedConvergence}, we may conclude,
			\[ \varphi * \psi = \lim_{k \rightarrow \infty} L_{m_{k}, n_{k}, \gamma_{k}} = \lim_{k \rightarrow \infty} \int_{\R^{d}} e_{m_{k}, n_{k}, \gamma_{k}}(x) dx = \int_{\R^{d}} \psi(x) T_{x} \varphi dx . \] 
		Finally, the continuity is a direct consequence of \eqref{eq:ConvGSEstimate} in view of \eqref{eq:GSasL1Sp}.
	\end{proof}
	
The \emph{short-time Fourier transform (STFT)} of a function $\varphi \in \mathcal{S}^{[M]}_{[A]}$ with respect to a window $\psi \in \mathcal{S}^{[M]}_{[A]}$ is given by\footnote{Our definition of the STFT is in a ``free $x$ variable way'' as a convolution. This will be convenient in the sequel, as we will extend convolution, and thus the STFT, to other spaces that are translation-modulation invariant.}
	\[ V_{\psi} \varphi(\xi) = (M_{- \xi} \varphi) * \check{\overline{\psi}} \in \mathcal{S}^{[M]}_{[A]} . \]
The \emph{adjoint STFT} of a function $\Phi \in C_{[M]}(\mathcal{S}^{[M]}_{[A]})$ with respect to a window $\gamma \in \mathcal{S}^{[M]}_{[A]}$ is defined by the Pettis integral
	\[ V^{*}_{\gamma} \Phi = \int_{\R^{d}} M_{\xi} [\Phi(\xi) * \gamma] d\xi . \]
	
	\begin{proposition}
		\label{p:STFTGS}
		Take $\psi_{0}, \psi_{1}, \gamma \in \mathcal{S}^{(M)}_{(A)}$ and write $\psi = \psi_{0} * \psi_{1}$. The linear mappings
			\[ V_{\psi} : \mathcal{S}^{[M]}_{[A]} \rightarrow C_{[M]}(\mathcal{S}^{[M]}_{[A]}) \quad \text{and} \quad V^{*}_{\gamma} : C_{[M]}(\mathcal{S}^{[M]}_{[A]}) \rightarrow \mathcal{S}^{[M]}_{[A]} , \]
		are well-defined and continuous. Moreover, if $(\gamma, \psi)_{L^{2}} \neq 0$, then, 
			\begin{equation}
				\label{eq:STFTGSReconstruct} 
				\frac{1}{(\gamma, \psi)_{L^{2}}} V^{*}_{\gamma} \circ V_{\psi} = \id_{\mathcal{S}^{[M]}_{[A]}} . 
			\end{equation}
	\end{proposition}
	
	\begin{proof}
		We first consider $V_{\psi}$. For any $q, \ell > 0$, take arbitrary $\varphi \in \mathcal{S}^{M, 1 / (2 q \sqrt{d})}_{A, 1, 2/\ell}$, $\psi_{0}, \psi_{1} \in \mathcal{S}^{(M)}_{(A)}$, and put $\psi = \psi_{0} * \psi_{1}$. Then, for any $\kappa \in \N^{d}$, by \eqref{eq:M1product}
			\begin{align*}
				&\norm{\xi^{\kappa} V_{\psi} \varphi(\xi)}_{\mathcal{S}^{M, \ell}_{A, 1/\ell}} \\
				&\leq \sup_{(\alpha, x) \in \N^{d} \times \R^{d}} (2 \pi)^{-|\kappa|} \sum_{\delta \leq \kappa} {\kappa \choose \delta} \int_{\R^{d}} \frac{|\psi^{(\alpha + \kappa - \delta)}(x)| e^{\omega_{A}(2 x / \ell)}}{\ell^{|\alpha|} M_{\alpha}} |\varphi^{(\delta)}(t - x)| e^{\omega_{A}(2(t - x) / \ell)} dx \\
				&\leq \sup_{(\alpha, x) \in \N^{d} \times \R^{d}} M_{0} \left( \frac{1}{q \sqrt{d}} \right)^{|\kappa|} M_{\kappa} \\ &\qquad \qquad \qquad \qquad \sum_{\delta \leq \kappa} \frac{{\kappa \choose \delta}}{2^{|\kappa|}} \int_{\R^{d}} \frac{|(\psi_{0}^{(\alpha)} *  \psi_{1}^{(\kappa - \delta)})(x)| e^{\omega_{A}(2x/\ell)}}{\ell^{|\alpha|} M_{\alpha} (1 / (2 q \sqrt{d}))^{|\kappa - \delta|} M_{\kappa - \delta}} \frac{|\varphi^{(\delta)}(t - x)| e^{\omega_{A}(2(t - x)/\ell)}}{(1 / (2 q \sqrt{d}))^{|\delta|} M_{\delta}} dx \\
				&\leq M_{0} \norm{\psi_{0}}_{\mathcal{S}^{M, \ell}_{A, 4/\ell}} \norm{\psi_{1}}_{\mathcal{S}^{M, 1 / (2 q \sqrt{d})}_{A, 1, 4 / \ell}} \norm{\varphi}_{\mathcal{S}^{M, 1 / (2 q \sqrt{d})}_{A, 1, 2 / \ell}} \left( \frac{1}{q \sqrt{d}} \right)^{|\kappa|} M_{\kappa} . 
			\end{align*}
		Hence, if we put $C = M_{0}^{2} \norm{\psi_{0}}_{\mathcal{S}^{M, \ell}_{A, 4/\ell}} \norm{\psi_{1}}_{\mathcal{S}^{M, 1 / (2 q \sqrt{d})}_{A, 1, 4 / \ell}}$, then,
			\[ \norm{V_{\psi} \varphi(\xi)}_{\mathcal{S}^{M, \ell}_{A, 1 / \ell}} \leq C \norm{\varphi}_{\mathcal{S}^{M, 1 / (2 q \sqrt{d})}_{A, 1, 2 / \ell}} \inf_{\kappa \in \N^{d}} \left( \frac{1}{q |\xi|} \right)^{|\kappa|} \frac{M_{\kappa}}{M_{0}} = C \norm{\varphi}_{\mathcal{S}^{M, 1 / (2 q \sqrt{d})}_{A, 1, 2 / \ell}} e^{- \omega_{M}( q \xi)} .  \]
		In view of \eqref{eq:GSasL1Sp}, it follows that $V_{\psi} : \mathcal{S}^{[M]}_{[A]} \rightarrow C_{[M]}(\mathcal{S}^{[M]}_{[A]})$ is a well-defined continuous linear map.
		
		Next, we treat $V^{*}_{\gamma}$. Suppose $\gamma \in \mathcal{S}^{(M)}_{(A)}$ and take any $\varphi \in \mathcal{S}^{M, \ell / 2}_{A, 2 / \ell}$ for arbitrary $\ell > 0$. For any $\xi \in \R^{d}$, we have by \eqref{eq:M1product} and \eqref{eq:ConvGSEstimate},
			\[ \norm{M_{\xi}[\varphi * \gamma]}_{\mathcal{S}^{M, \ell}_{A, 1/\ell}} \leq M_{0} \sum_{\beta \leq \alpha} \frac{{\alpha \choose \beta}}{2^{|\alpha|}} \frac{|2 \pi \xi|^{\beta}}{(\ell / 2)^{\beta} M_{\beta}} \norm{\varphi * \gamma}_{\mathcal{S}^{M, \ell/2}_{A, 1/\ell}} \leq e^{\omega_{M}(4 \pi \xi / \ell)} \norm{\varphi}_{\mathcal{S}^{M, \ell / 2}_{A, 2/\ell}} \norm{\gamma}_{\mathcal{S}^{M, \ell / 2}_{A, 1, 2/\ell}} . \]
		Now take any bounded set $B \subset C_{[M]}(\mathcal{S}^{[M]}_{[A]})$, then $B$ is a bounded subset of the space $C_{e^{\omega_{M}(4 \pi H^{d + 1} \cdot / \ell)}}(\mathcal{S}^{[M]}_{[A]})$ for any $\ell > 0$ (for some $\ell > 0$). In particular, $\{ \Phi(\xi) e^{\omega_{M}(4 \pi H^{d + 1} \xi / \ell)} : \Phi \in B, \xi \in \R^{d} \}$ is a bounded set in $\mathcal{S}^{[M]}_{[A]}$. Consequently, we find the constant $C = \sup_{\Phi \in B, \xi \in \R^{d}} e^{\omega_{M}(4 \pi H^{d + 1} \xi / \ell)} \norm{\Phi(\xi)}_{\mathcal{S}^{M, \ell/2}_{A, 2/\ell}} < \infty$, where in the Roumieu case we had to possibly raise $\ell$. We get 
			\[ \norm{M_{\xi}[\Phi(\xi) * \gamma]}_{\mathcal{S}^{M, \ell}_{A, 1/\ell}} \leq C \norm{\gamma}_{\mathcal{S}^{M, \ell/2}_{A, 1, 2/\ell}} \frac{e^{\omega_{M}(4 \pi \xi / \ell)}}{e^{\omega_{M}(4 \pi H^{d + 1} \xi / \ell)}} \in L^{1} , \qquad \forall \Phi \in B, \xi \in \R^{d} .  \]
		By \eqref{eq:UppdarBoundSeminormPettisIntegral} it follows that $\{ V^{*}_{\gamma} \Phi : \Phi \in B \}$ is a bounded subset of $\mathcal{S}^{[M]}_{[A]}$. As by Lemma \ref{l:TopCMSp} the space $C_{[M]}(\mathcal{S}^{[M]}_{[A]})$ is Fr\'{e}chet (an $(LB)$-space as $\mathcal{S}^{\{M\}}_{\{A\}}$ is compactly regular), it is bornological. Hence we may conclude that $V^{*}_{\gamma} : C_{[M]}(\mathcal{S}^{[M]}_{[A]}) \rightarrow \mathcal{S}^{[M]}_{[A]}$ is a well-defined continuous linear map. 
		
		Finally, we note that \eqref{eq:STFTGSReconstruct} holds, since $\mathcal{S}^{[M]}_{[A]}$ is a subspace of $L^{2}(\R^{d})$ \cite[Theorem 1.2]{D-N-V-NuclGSKerThm} and the identity is valid there \cite[Corollary 3.2.3]{G-FoundationsTimeFreqAnal}. 
	\end{proof}
	
In view of the condition that the window $\psi$ has to be exactly the convolution of two elements in $\mathcal{S}^{(M)}_{(A)}$, we demonstrate that one may always choose a $\psi$ that enjoys nice properties for our purposes.

	\begin{lemma}
		\label{l:AdequateWindow}
		There exist $\psi_{0}, \psi \in \mathcal{S}^{(M)}_{(A)}$ such that $\psi = \psi_{0} * \psi_{0}$ and $(\psi, \psi)_{L^{2}} = 1$. 
	\end{lemma}
	
	\begin{proof}
		Take any $\phi \in \mathcal{S}^{(M)}_{(A)}$ such that $\phi(0) \neq 0$. We may assume that $\phi$ is even, as otherwise we could just take $\phi \cdot \check{\phi} \in \mathcal{S}^{(M)}_{(A)}$. Now set $\varphi = \phi * \phi \in \mathcal{S}^{(M)}_{(A)}$ and note that $\varphi(0) = \int_{\R^{d}} |\phi(t)|^{2} dt \neq 0$. Put $\lambda = (\varphi, \varphi)_{L^{2}} > 0$. Choosing $\psi_{0} = \lambda^{-1/4} \phi$, then $\psi = \psi_{0} * \psi_{0}$ will have the desired properties.
	\end{proof}
	
\subsection{The tempered ultradistributions}
	
We denote $\mathcal{S}^{\prime [M]}_{[A]}$ for the strong dual of the Gelfand-Shilov space $\mathcal{S}^{[M]}_{[A]}$, also called the space of \emph{tempered ultradistributions (of Beurling and Roumieu type)}. Then, $\mathcal{S}^{\prime (M)}_{(A)}$ is a $(DFN)$-space, while $\mathcal{S}^{\prime \{M\}}_{\{A\}}$ is a $(FN)$-space. For any $\ell > 0$, we will often use the following spaces
	\begin{equation}
		\label{eq:XSpaces}
		X_{(\ell)} = \overline{\mathcal{S}^{(M)}_{(A)}}^{\mathcal{S}^{M, \ell}_{A, \ell^{-1}}} , \qquad X_{\{\ell\}} = \mathcal{S}^{M, \ell}_{A, \ell^{-1}} . 
	\end{equation} 
Then, as locally convex spaces,
	\[ \mathcal{S}^{\prime (M)}_{(A)} = \varinjlim_{\ell \rightarrow \infty} X^{\prime}_{(\ell)} , \qquad \mathcal{S}^{\prime \{M\}}_{\{A\}} = \varprojlim_{\ell \rightarrow \infty} X^{\prime}_{\{\ell\}} . \]

We define the convolution of a tempered ultradistribution $f \in \mathcal{S}^{\prime [M]}_{[A]}$ and a test function $\psi \in \mathcal{S}^{[M]}_{[A]}$ by transposition, i.e.,
	\[ \ev{f * \psi}{\varphi} = \ev{f}{\varphi * \check{\psi}} , \qquad \forall \varphi \in \mathcal{S}^{[M]}_{[A]} . \]
Alternatively, we may express $f * \psi$ as a Pettis integral.

	\begin{lemma}
		For any $f \in \mathcal{S}^{\prime [M]}_{[A]}$ and $\psi \in \mathcal{S}^{[M]}_{[A]}$, we have that
			\begin{equation} 
				\label{eq:ConvTempUltraDistPettis}
				f * \psi = \int_{\R^{d}} \psi(x) T_{x} f dx , 
			\end{equation}
		where the right-hand side is a Pettis integral. In particular, the convolution map
			\begin{equation}
				\label{eq:ConvTempUltraDist} 
				* : \mathcal{S}^{\prime [M]}_{[A]} \times \mathcal{S}^{[M]}_{[A]} \rightarrow \mathcal{S}^{\prime [M]}_{[A]} : \quad (f, \psi) \mapsto \int_{\R^{d}} \psi(x) T_{x} f dx  
			\end{equation}
		is a well-defined hypocontinuous bilinear map.
	\end{lemma}
	
	\begin{proof}
		The hypocontinuity of \eqref{eq:ConvTempUltraDist} follows directly from that of \eqref{eq:GSConv}. Now, let $X_{[\ell]}$ be as in \eqref{eq:XSpaces}, then $f \in X^{\prime}_{[\ell_{0}]}$ for some $\ell_{0} > 0$ (for any $\ell_{0} > 0$). Note that for any $\varphi \in \mathcal{S}^{[M]}_{[A]}$ we have $\ev{\psi(x) T_{x} f}{\varphi} = \ev{f}{\check{\psi}(-x)T_{-x}\varphi}$. Suppose $B$ is a bounded subset of $\mathcal{S}^{[M]}_{[A]}$, then in particular $B$ is contained and bounded in any (in some) $X_{[\ell]}$. Consequently, we may assume the set $\widetilde{B} = \{ T_{-x} \varphi / e^{\omega_{A}(-2x/\ell_{0})} : \varphi \in B, x \in \R^{d} \}$ is bounded in $X_{[\ell_{0}]}$ and that $\psi \in L^{1}_{e^{\omega_{A}(2 \cdot / \ell_{0})}}$, so we get
			\[ \int_{\R^{d}} \sup_{\varphi \in B} |\ev{\psi(x) T_{x} f}{\varphi}| dx \leq \sup_{\varphi \in \widetilde{B}} |\ev{f}{\varphi}| \norm{\psi}_{L^{1}_{e^{\omega_{A}(2 \cdot / \ell_{0})}}} . \]	
		By Lemma \ref{l:SuffCondExistencePettisIntegral} we may now conclude that the Pettis integral $\int_{\R^{d}} \psi(x) T_{x} fdx$ exists. Also, by applying \eqref{eq:LinearMapPettisIntegral} twice, we see that for any $\varphi \in \mathcal{S}^{[M]}_{[A]}$,
			\[ \ev{\int_{\R^{d}} \psi(x) T_{x} fdx}{\varphi} = \int_{\R^{d}} \ev{\psi(x) T_{x} f}{\varphi} dx = \int_{\R^{d}} \ev{f}{\check{\psi}(-x)T_{-x}\varphi} dx = \ev{f}{\varphi * \check{\psi}} . \]
	\end{proof}

By the Pettis integral representation, we see that the convolution \eqref{eq:ConvTempUltraDist} is the (unique) extension of the convolution \eqref{eq:GSConv}. In particular, we may now extend the (adjoint) STFT to the tempered ultradistributions. For any $f \in \mathcal{S}^{\prime [M]}_{[A]}$ we define the STFT with respect to the window $\psi \in \mathcal{S}^{[M]}_{[A]}$ as
	\[ V_{\psi} f(\xi) = (M_{- \xi} f) * \check{\overline{\psi}} \in \mathcal{S}^{\prime [M]}_{[A]} . \] 
The adjoint STFT of a function $\Phi \in C_{[M], d}(\mathcal{S}^{\prime [M]}_{[A]})$ with respect to a window $\gamma \in \mathcal{S}^{[M]}_{[A]}$ is defined by the Pettis integral
	\[ V^{*}_{\gamma} \Phi = \int_{\R^{d}} M_{\xi} [\Phi(\xi) * \gamma] d\xi . \]

	\begin{proposition}
		\label{p:STFTTempUltra}
		Take $\psi, \gamma_{0}, \gamma_{1} \in \mathcal{S}^{(M)}_{(A)}$ and put $\gamma = \gamma_{0} * \gamma_{1}$. The linear mappings
			\[ V_{\psi} : \mathcal{S}^{\prime [M]}_{[A]} \rightarrow C_{[M], d}(\mathcal{S}^{\prime [M]}_{[A]}) \quad \text{and} \quad V^{*}_{\gamma} : C_{[M], d}(\mathcal{S}^{\prime [M]}_{[A]}) \rightarrow \mathcal{S}^{\prime [M]}_{[A]} , \]
		are well-defined and continuous. Moreover, if $(\gamma, \psi)_{L^{2}} \neq 0$, then, 
			\begin{equation}
				\label{eq:STFTReconstructTempUltra} 
				\frac{1}{(\gamma, \psi)_{L^{2}}} V^{*}_{\gamma} \circ V_{\psi} = \id_{\mathcal{S}^{\prime [M]}_{[A]}} 
			\end{equation}
		is valid and the desingularization formula
			\begin{equation}
				\label{eq:STFTDesingularizationTempUltra}
				\ev{f}{\varphi} = \frac{1}{(\gamma, \psi)_{L^{2}}} \int_{\R^{d}} \ev{V_{\psi} f(\xi)}{V_{\overline{\gamma}} \varphi(-\xi)} d\xi
			\end{equation}
		holds for all $f \in \mathcal{S}^{\prime [M]}_{[A]}$ and $\varphi \in \mathcal{S}^{[M]}_{[A]}$. 
	\end{proposition}
	
	\begin{proof}
		We first show the continuity of $V_{\psi}$. Take any $f \in \mathcal{S}^{\prime [M]}_{[A]}$. Let $X_{[\ell]}$ be as in \eqref{eq:XSpaces}, then for some $\ell_{0} > 0$ (for any $\ell_{0} > 0$) $f$ is an element $X^{\prime}_{[\ell_{0}]}$. Now, take any bounded subset $B$ of $\mathcal{S}^{[M]}_{[A]}$. Then, for any $\ell_{1} > 0$ (for some $\ell_{1} > 0$), $B$ is a bounded subset of $\mathcal{S}^{M, \ell_{1}}_{A, 1/\ell_{1}}$ and we set $C_{B} = \sup_{\varphi \in B} \norm{\varphi}_{\mathcal{S}^{M, \ell_{1}}_{A, 1/\ell_{1}}}$. For $q > 0$, we have that by \eqref{eq:M1product}
			\begin{align*} 
				\| M_{-\xi} [\varphi * \overline{\psi}] \|_{\mathcal{S}^{M, 4\pi / q}_{A, 1 / (2\ell_{1})}} 
				&\leq \sup_{(\alpha, x) \in \N^{d} \times \R^{d}} \sum_{\beta \leq \alpha} {\alpha \choose \beta} \frac{(2 \pi |\xi|)^{|\beta|} M_{0}}{(4 \pi / q)^{|\beta|} M_{\beta}} \frac{|\varphi * \overline{\psi} \vspace{0.1mm}^{(\alpha - \beta)}(x)| e^{\omega_{A}(x / (2 \ell_{1}))}}{(4 \pi / q)^{|\alpha - \beta|} M_{\alpha - \beta}} \\
				&\leq C_{B} M_{0} \norm{\psi}_{\mathcal{S}^{M, 2 \pi / q}_{A, 1, 1 / \ell_{1}}} e^{\omega_{M}(q \xi)} .
			\end{align*}
		Consequently, if we have $\ell_{0} = \max(4 \pi / q, 2 \ell_{1})$, it follows that
			\[ \sup_{\varphi \in B} |\ev{V_{\psi} f(\xi)}{\varphi}| e^{- \omega_{M}(q \xi)} \leq C \norm{f}_{X^{\prime}_{[\ell_{0}]}} , \]
		for $C = C_{B} M_{0} \norm{\psi}_{\mathcal{S}^{M, 2 \pi / q}_{A, 1, 1 / \ell_{1}}}$. The continuity of $V_{\psi}$ now follows.  
		
		Next, we consider the continuity of $V^{*}_{\gamma}$. Note that for any $f \in \mathcal{S}^{\prime [M]}_{[A]}$ and $\varphi \in \mathcal{S}^{[M]}_{[A]}$
			\[ \ev{M_{\xi} [f * \gamma]}{\varphi} = \ev{f}{(M_{\xi} \varphi) * \check{\gamma}} = \ev{f}{V_{\overline{\gamma}} \varphi(-\xi)} .  \]
		As a result of Lemma \ref{l:InducedBilinearMap}, with $A = \ev{\cdot}{\cdot}$ and $\Gamma_{\xi} \equiv \id$, and Proposition \ref{p:STFTGS}, we find the following well-defined hypocontinuous bilinear map,
			\[ \mathcal{S}^{[M]}_{[A]} \times C_{[M], d}(\mathcal{S}^{\prime [M]}_{[A]}) \rightarrow \C : \quad (\varphi, \Phi) \mapsto \int_{\R^{d}} \ev{\Phi(\xi)}{V_{\overline{\gamma}} \varphi(-\xi)} d\xi . \]
		Then, the righthand side will exactly be $\ev{V^{*}_{\gamma} \Phi}{\varphi}$, so that $V^{*}_{\gamma} \Phi$ is indeed an element of $\mathcal{S}^{\prime [M]}_{[A]}$. Moreover, for any bounded subset $B \subset C_{[M], d}(\mathcal{S}^{\prime [M]}_{[A]})$, the set $\{ V^{*}_{\gamma} \Phi : \Phi \in B \}$ is an equicontinuous subset of $\mathcal{S}^{\prime [M]}_{[A]}$, and thus bounded. By Lemma \ref{l:TopCMSp} $C_{[M], d}(\mathcal{S}^{\prime [M]}_{[A]})$ is an $(LB)$-space (a Fr\'{e}chet space) since $\mathcal{S}^{\prime (M)}_{(A)}$ is compactly regular, hence it is bornological. The continuity of $V^{*}_{\gamma}$ now follows.
		
		Now, suppose $(\gamma, \psi)_{L^{2}} \neq 0$. As $\mathcal{S}^{[M]}_{[A]} \subset L^{2}(\R^{d}) \subset \mathcal{S}^{\prime [M]}_{[A]}$ and $\mathcal{S}^{[M]}_{[A]}$ is dense in $\mathcal{S}^{\prime [M]}_{[A]}$, it follows that $L^{2}(\R^{d})$ is dense in $\mathcal{S}^{\prime [M]}_{[A]}$. Then \eqref{eq:STFTReconstructTempUltra} follows directly from \cite[Corollary 3.2.3]{G-FoundationsTimeFreqAnal}. By Lemma \ref{l:InducedBilinearMap}, for any $\varphi \in \mathcal{S}^{[M]}_{[A]}$, the linear map
			\[ f \mapsto \int_{\R^{d}} \ev{V_{\psi} f(\xi)}{V_{\overline{\gamma}} \varphi(-\xi)} d\xi , \]
		is continuous over $\mathcal{S}^{\prime [M]}_{[A]}$. Again, by the density of $L^{2}(\R^{d})$ in $\mathcal{S}^{\prime [M]}_{[A]}$, it follows from \cite[Theorem 3.2.1]{G-FoundationsTimeFreqAnal} that the mapping above coincides with the evaluation in $\varphi$. Hence \eqref{eq:STFTDesingularizationTempUltra} holds. 
	\end{proof}

\section{Translation-Modulation invariant Banach spaces and the space $\D^{[M]}_{E}$}
\label{sec:TMIB}

We now introduce the class of Banach spaces of tempered ultradistributions that are invariant under translation and modulation, which will form the primary building blocks for the theory to come. Throughout this section $M$ and $A$ denote two weight sequences both satisfying $(M.1)$ and $(M.2)'$. 

	\begin{definition}
		A Banach space $(E, \|\cdot\|_{E})$ is called a \emph{translation-modulation invariant Banach space of ultradistributions} (in short: TMIB) of class $[M] - [A]$ if it satisfies the following three conditions:
			\begin{itemize}
				\item[$(a)$] the continuous and dense inclusions $\mathcal{S}^{[M]}_{[A]} \hookrightarrow E \hookrightarrow \mathcal{S}^{\prime [M]}_{[A]}$ hold;
				\item[$(b)$] $T_{x}(E) \subseteq E$ and $M_{\xi}(E) \subseteq E$ for all $x, \xi \in \R^{d}$;
				\item[$(c)$] there exist $q, C_{E, q} > 0$ (for every $q > 0$ there exists a $C_{E, q} > 0$) such that for all $x, \xi \in \R^{d}$:
					\begin{equation}
						\label{eq:TMIBnormcond}
						\omega_{E}(x) := \norm{T_{x}}_{L(E)} \leq C_{E, q} e^{\omega_{A}(q x)} \quad \text{and} \quad \nu_{E}(\xi) := \norm{M_{-\xi}}_{L(E)} \leq C_{E, q} e^{\omega_{M}(q \xi)} .
					\end{equation}
			\end{itemize}
			
		If $E$ is a Banach space such that $E = E_{0}^{\prime}$ for some TMIB $E_{0}$ of class $[M] - [A]$, then $E$ is called a \emph{dual translation-modulation invariant Banach space of ultradistributions} (in short: DTMIB) of class $[M] - [A]$.
	\end{definition}
	
	\begin{remark}
		In \eqref{eq:TMIBnormcond} we implicitly use that $T_{x}$ and $M_{\xi}$ are continuous linear operators on $E$. This is a direct consequence of $(a)$, $(b)$, and the closed graph theorem, see \cite[Lemma 3.1]{D-P-V-ClassTISpQuasiAnalUltraDist}. 
	\end{remark}
	
	\begin{remark}
		\label{r:dualiswTMIB}
		Let $E = E^{\prime}_{0}$ be a DTMIB. Then clearly $(b)$ holds via transposition. Moreover, the bipolar theorem yields that $\omega_{E} = \check{\omega}_{E_{0}}$ and $\nu_{E} = \nu_{E_{0}}$, so that $E$ also satisfies $(c)$. As $E_{0}$ satisfies the condition $(a)$, we find $\mathcal{S}^{[M]}_{[A]} \subset E \hookrightarrow \mathcal{S}^{\prime [M]}_{[A]}$, where the latter dense inclusion follows from the fact that $\mathcal{S}^{[M]}_{[A]}$ is dense in $\mathcal{S}^{\prime [M]}_{[A]}$. However, the continuous inclusion $\mathcal{S}^{[M]}_{[A]} \subset E$ is not necessarily dense. A classical example of a DTMIB that is not a TMIB is $L^{\infty}$. If $E_{0}$ is reflexive, then $E$ is always a TMIB due to the continuous inclusion $E_{0} \subset \mathcal{S}^{\prime [M]}_{[A]}$. 
	\end{remark}
	
For a TMIB or DTMIB $E$, the functions $\omega_{E}$ and $\nu_{E}$ are both positive (since $T_{x} f \neq 0$ and $M_{\xi} f \neq 0$ for any $f \in \mathcal{S}^{\prime [M]}_{[A]} \setminus \{0\}$) and measurable (cfr. \cite[p. 149]{D-P-V-ClassTISpQuasiAnalUltraDist} and Remark \ref{r:dualiswTMIB}). By \eqref{eq:TMIBnormcond} it follows that $\omega_{E}$ and $\nu_{E}$ are locally bounded. On the other hand, let $\omega_{E}(x) \leq C_{E, q} e^{\omega_{A}(q x)}$ for some $q > 0$ and take any $e \in E$ such that $\norm{e}_{E} = 1 / C_{E, q}$. Let $U$ be an open ball around the origin on which $\omega_{A}$ is identically zero, then note that the elements of $\{ T_{y} e : y \in U \}$ have norm at most $1$ in $E$. For a fixed $x_{0} \in \R^{d}$, we now see that for any $x \in x_{0} + U$,
	\[ \omega_{E}(x) = \norm{T_{x}}_{E} \geq \sup_{y \in U} \norm{T_{x} T_{y} e}_{E} \geq \norm{T_{x_{0}} e}_{E} .  \]
Consequently, we see that $\omega_{E}^{-1}$ is locally bounded, and similarly one can show that $\nu_{E}^{-1}$ is locally bounded. Hence, it makes sense to consider the Banach space $L^{1}_{\omega_{E}}$. A strong property of $E$ is that one may naturally define convolution with elements of $L^{1}_{\omega_{E}}$ on it.
	
	\begin{proposition}
		\label{p:ExtensionOfConvolutionToWTMIB}
		Let $E$ be a TMIB or DTMIB of class $[M] - [A]$. The mapping
			\begin{equation} 
				\label{eq:ConvWTMIB}
				* : E \times L^{1}_{\omega_{E}} \rightarrow E : \quad (f, g) \mapsto \int_{\R^{d}} g(x) T_{x} f dx ,
			\end{equation}
		where the righthand side is represented as a Pettis integral, is a well-defined and continuous bilinear mapping, that turns $E$ into a Banach module over $L^{1}_{\omega_{E}}$, that is,
			\begin{equation}
				\label{eq:convolutionineqE}
				\norm{f * g}_{E} \leq \norm{f}_{E} \norm{g}_{L^{1}_{\omega_{E}}} , \qquad f \in E, g \in L^{1}_{\omega_{E}} . 
			\end{equation}
	\end{proposition}
	
	\begin{proof}
		We have, for any $f \in E$ and $g \in L^{1}_{\omega_{E}}$,
			\[ \int_{\R^{d}} \norm{g(x) T_{x} f}_{E} dx \leq \int_{\R^{d}} |g(x)| \omega_{E}(x) \norm{f}_{E} dx = \norm{f}_{E} \norm{g}_{L^{1}_{\omega_{E}}} , \]
		hence the function is Pettis integrable by Lemma \ref{l:SuffCondExistencePettisIntegral}. From here the continuity and \eqref{eq:convolutionineqE} follow directly.
	\end{proof}
	
	\begin{remark}
		The restricted convolution $* : E \times \mathcal{S}^{[M]}_{[A]} \rightarrow E$ is extended by the convolution \eqref{eq:ConvTempUltraDist} in $\mathcal{S}^{\prime [M]}_{[A]}$. In particular, if $f \in \mathcal{S}^{[M]}_{[A]}$ and $\psi \in \mathcal{S}^{[M]}_{[A]}$, then the definitions of $f * \psi$ as in \eqref{eq:GSConvPettis} and \eqref{eq:ConvWTMIB} coincide. In the case where $E$ is a TMIB the extension of the convolution \eqref{eq:GSConv} is also unique. This was already shown, under the condition $(M.2)$, in \cite[Proposition 3.2]{D-P-P-V-TMIBUltraDist}. 
	\end{remark}

We now consider a particular subspace of a TMIB or DTMIB $E$. For any $\ell > 0$ we define the Banach space $\D^{M, \ell}_{E}$ of all $f \in \mathcal{S}^{\prime [M]}_{[A]}$ such that $f^{(\alpha)} \in E$ for any $\alpha \in \N^{d}$ and
	\[ \norm{f}_{\D^{M, \ell}_{E}} = \sup_{\alpha \in \N^{d}} \frac{\norm{f^{(\alpha)}}_{E}}{\ell^{|\alpha|} M_{\alpha}} < \infty . \]
Next, we put
	\[ \D^{(M)}_{E} = \varprojlim_{\ell \rightarrow 0^{+}} \D^{M, \ell}_{E} , \qquad D^{\{M\}}_{E} = \varinjlim_{\ell \rightarrow \infty} D^{M, \ell}_{E} . \]
These spaces contain all convolutions between elements of $E$ and $\mathcal{S}^{[M]}_{[A]}$.
	
	\begin{lemma}
		\label{l:ConvolutionEtoDE}
		The mapping
			\[ * : E \times \mathcal{S}^{[M]}_{[A]} \rightarrow \D^{[M]}_{E} : \quad (f, \psi) \mapsto f * \psi , \]
		is well-defined and continuous.
	\end{lemma}
	
	\begin{proof}
		Take any $f \in E$ and $\psi \in \mathcal{S}^{[M]}_{[A]}$. Let $q, \ell > 0$ be such that $\psi \in \mathcal{S}^{M, \ell}_{A, 1, q}$ and $\omega_{E}(x) \leq C_{E, q} e^{\omega_{A}(q x)}$. By \eqref{eq:convolutionineqE} we have that, for any $\alpha \in \N^{d}$,
			\[ \frac{\norm{f * \psi^{(\alpha)}}_{E}}{\ell^{|\alpha|} M_{\alpha}} \leq \frac{\norm{f}_{E} \norm{\psi^{(\alpha)}}_{L^{1}_{\omega_{E}}}}{{\ell^{|\alpha|} M_{\alpha}} } \leq C_{E, q} \norm{f}_{E} \norm{\psi}_{\mathcal{S}^{M, \ell}_{A, 1, q}} .  \]
		Hence $\norm{f * \psi}_{\D^{M, \ell}_{E}} \leq C_{E, q} \norm{f}_{E} \norm{\psi}_{\mathcal{S}^{M, \ell}_{A, 1, q}}$, and the continuity follows by \eqref{eq:GSasL1Sp}.   
	\end{proof}
	
An immediate corollary of the previous convolution mapping is the density of $\D^{[M]}_{E}$ in $E$, as shown in the next result. 

	\begin{corollary}
		\label{c:DEdenseinE}
		The space $\D^{[M]}_{E}$ is dense in $E$.
	\end{corollary}
	
	\begin{proof}
		Take any $\chi \in \mathcal{S}^{[M]}_{[A]}$ such that $\int_{\R^{d}} \chi(x) dx = 1$. Put $\chi_{n} = n^{d} \chi(n \cdot)$ and $\psi_{n} = \chi * \chi_{n}$ for any $n \geq 1$. Note that $\psi_{n} \in \mathcal{S}^{[M]}_{[A]}$ by Lemma \ref{l:ConvolutionGS} and $\int_{\R^{d}} \psi_{n}(x) dx = 1$ by Fubini's Theorem. Let $q, \ell > 0$ be such that $\chi \in \mathcal{S}^{M, \ell}_{A, 2 q} \cap \mathcal{S}^{M, \ell}_{A, 1, 2q}$ and $\omega_{E}(x) \leq C_{E, q/H^{d+1}} e^{\omega_{A}(q x / H^{d + 1})}$. By \eqref{eq:ConvGSEstimate} we find
			\[ \norm{\psi_{n}}_{\mathcal{S}^{M, \ell}_{A, q}} \leq \norm{\chi}_{\mathcal{S}^{M, \ell}_{A, 2q}} \int_{\R^{d}} |\chi_{n}(x)| e^{\omega_{A}(2qx)} dx \leq \norm{\chi}_{\mathcal{S}^{M, \ell}_{A, 2q}} \norm{\chi}_{\mathcal{S}^{M, \ell}_{A, 1, 2q}} . \]
		In particular, we have $|\psi_{n}(x)| \leq C e^{-\omega_{A}(q x)}$ for every $n \geq 1$ and some $C > 0$. For arbitrary $\varphi \in E$ we have that $\varphi * \psi_{n} \in \D^{[M]}_{E}$ by Lemma \ref{l:ConvolutionEtoDE}. Moreover, we see that
			\[ \norm{\psi_{n}(x) [\varphi - T_{x} \varphi]}_{E} \leq 2 C C_{E, q / H^{d + 1}} \norm{\varphi}_{E} \frac{e^{\omega_{A}(q x / H^{d + 1})}}{e^{\omega_{A}(q x)}} \in L^{1} . \]
		Then, Lemma \ref{l:PettisIntegralsDominatedConvergence} gives
			\[ \lim_{n \rightarrow \infty} \varphi - \varphi * \psi_{n} = \lim_{n \rightarrow \infty} \int_{\R^{d}} \psi_{n}(x) [\varphi - T_{x} \varphi] dx = 0 . \]
	\end{proof}
	
In the sequel, we will often work with the following extra condition on modulation in $E$:
	\begin{equation}
		\label{eq:GrowthCondModulation}
		\forall q_{0} > 0 ~ \exists q_{1} > 0 ~ (\forall q_{1} > 0 ~ \exists q_{0} > 0) ~ \exists C > 0 ~ \forall \xi \in \R^{d} : ~ \nu_{E}(\xi) e^{\omega_{M}(q_{0} \xi)} \leq C e^{\omega_{M}(q_{1} \xi)} . 
	\end{equation}
Note that if $E$ satisfies \eqref{eq:GrowthCondModulation}, then in particular the family of mappings $M_{\xi} : E \rightarrow E$ satisfy \eqref{eq:InducedBilinearMapCond}. Also, if $E = E_{0}^{\prime}$ is a DTMIB and $E_{0}$ satisfies \eqref{eq:GrowthCondModulation}, then so does $E$ by Remark \ref{r:dualiswTMIB}.
	
	\begin{remark}
		\label{r:GrowthCondModulationSuffCond}
		The condition \eqref{eq:GrowthCondModulation} imposes stronger conditions on the weight sequence $M$, the TMIB or DTMIB $E$, or on both. It holds in the following cases:
			\begin{itemize}
				\item[$(i)$] if $\nu_{E}$ is a bounded function on $\R^{d}$;
				\item[$(ii)$] if $M$ satisfies $(M.2)'$ and $\sup_{\xi \in \R^{d}} \nu_{E}(\xi) / (1 + |\xi|)^{k} < \infty$ for some $k > 0$ by \eqref{eq:M2'};
				\item[$(iii)$] if $M$ satisfies $(M.2)$ by \cite[Proposition 3.6]{K-Ultradistributions1}.
			\end{itemize}
	\end{remark}
	
As the elements of $E$, and in particular those of $\D^{[M]}_{E}$, are contained in $\mathcal{S}^{\prime [M]}_{[A]}$, we may apply the STFT on them. Due to the specific behavior of the convolution on $E$, 	we are now able to establish the following mapping properties of the (adjoint) STFT on $\D^{[M]}_{E}$. 	
	
	\begin{proposition}
		\label{p:STFTDE}
		Suppose $E$ is a TMIB or DTMIB such that \eqref{eq:GrowthCondModulation} holds. Then, the mappings
			\begin{align*}
				V : \D^{[M]}_{E} \times \mathcal{S}^{[M]}_{[A]} \rightarrow C_{[M]}(E) : &\quad (f, \psi) \longmapsto V_{\psi} f \\
				V^{*} : C_{[M]}(E) \times \mathcal{S}^{[M]}_{[A]} \rightarrow \D^{[M]}_{E} : &\quad (\Phi, \gamma) \longmapsto V^{*}_{\gamma} \Phi 
			\end{align*}
		are well-defined and continuous.
	\end{proposition}	
	
	\begin{proof}
		Take any $f \in \D^{[M]}_{E}$ and $\psi \in \mathcal{S}^{[M]}_{[A]}$. Then, for any $q > 0$ let $\ell > 0$ be such that (for a fixed $\ell > 0$ let $q > 0$ be such that) $f \in \D^{M, \ell}_{E}$, $\psi \in \mathcal{S}^{M, \ell}_{A, 1, q}$, for which $\omega_{E}(x) \leq C_{E, \ell^{-1}} e^{\omega_{A}(x / \ell)}$, and $\nu_{E}(\xi) e^{\omega_{M}(q \xi)} \leq C e^{\omega_{M}(\frac{\pi \xi}{\ell \sqrt{d}})}$. We have by \eqref{eq:M1product},
			\begin{align*}
				|\xi^{\alpha}| \norm{V_{\psi} f(\xi)}_{E} &\leq \frac{1}{(2 \pi)^{|\alpha|}} \sum_{\beta \leq \alpha} {\alpha \choose \beta} \| (M_{- \xi} f^{(\beta)}) * \check{\overline{\psi}}{}^{(\alpha - \beta)} \|_{E} \\ 
				&\leq \frac{1}{(2 \pi)^{|\alpha|}} \sum_{\beta \leq \alpha} {\alpha \choose \beta} \| M_{-\xi} f^{(\beta)} \|_{E} \| \check{\overline{\psi}}{}^{(\alpha - \beta)} \|_{L^{1}_{\omega_{E}}} \\
				&\leq C_{E, \ell^{-1}} M_{0} \left(\frac{\ell}{\pi}\right)^{|\alpha|} M_{\alpha} \nu_{E}(\xi) \norm{f}_{\D^{M, \ell}_{E}} \norm{\psi}_{\mathcal{S}^{M, \ell}_{A, 1, \ell^{-1}}} , 
			\end{align*}
		whence, for $C' = C C_{E, \ell^{-1}} M_{0}^{2}$,
			\begin{align*} 
				\norm{V_{\psi} f(\xi)}_{E} &\leq \frac{C'}{C} \norm{f}_{\D^{M, \ell}_{E}} \norm{\psi}_{\mathcal{S}^{M, \ell}_{A, 1, \ell^{-1}}} \nu_{E}(\xi) \inf_{\alpha \in \N^{d}} \left( \frac{\ell \sqrt{d}}{\pi |\xi|} \right)^{|\alpha|} \frac{M_{\alpha}}{M_{0}} \\ 
				&= \frac{C'}{C} \norm{f}_{\D^{M, \ell}_{E}} \norm{\psi}_{\mathcal{S}^{M, \ell}_{A, 1, \ell^{-1}}} \nu_{E}(\xi) e^{-\omega_{M}\left(\frac{\pi \xi}{\ell \sqrt{d}}\right)} \\
				&\leq C' \norm{f}_{\D^{M, \ell}_{E}} \norm{\psi}_{\mathcal{S}^{M, \ell}_{A, 1, \ell^{-1}}} e^{- \omega_{M}(q \xi)} . 
			\end{align*}
		The continuity of $V$ now follows from \eqref{eq:GSasL1Sp}.
		
		Next, we show that $V^{*}$ is continuous. Take any $\Phi \in C_{[M]}(E)$ and $\gamma \in \mathcal{S}^{[M]}_{[A]}$. For any $\ell > 0$ let $q > 0$ be such that (for a fixed $q > 0$ let $\ell > 0$ be such that) $\Phi \in C_{e^{\omega_{M}(q \cdot)}}(E)$, $\gamma \in \mathcal{S}^{M, \ell / 2}_{A, 1, q}$, and for which $\omega_{E}(\xi) \leq C_{E, q} e^{\omega_{A}(q x)}$ and $\nu_{E}(\xi) e^{\omega_{M}(4 \pi H^{d + 1} \xi / \ell)} \leq C e^{\omega_{M}(q \xi)}$. Then, for any $\alpha \in \N^{d}$, by \eqref{eq:M1product} and \eqref{eq:M2'},
			\begin{align*}
				&\norm{\partial^{\alpha}(M_{\xi}[\Phi(\xi) * \gamma])}_{E} 
				\leq \sum_{\beta \leq \alpha} {\alpha \choose \beta} (2 \pi |\xi|)^{|\beta|} \norm{M_{\xi}[\Phi(\xi) * \gamma^{(\alpha - \beta)}]}_{E} \\
				&\qquad \leq C C_{E, q} \ell^{|\alpha|} M_{\alpha} \sum_{\beta \leq \alpha} \frac{{\alpha \choose \beta}}{2^{|\alpha|}} \frac{(4 \pi |\xi|)^{|\beta|} M_{0}}{\ell^{|\beta|} M_{\beta}} \frac{\norm{\gamma^{(\alpha - \beta)}}_{L^{1}_{e^{\omega_{A}(q \cdot)}}}}{(\ell / 2)^{|\alpha - \beta|} M_{\alpha - \beta}} \norm{\Phi}_{C_{e^{\omega_{M}(q \cdot)}}(E)} e^{- \omega_{M}(4 \pi H^{d + 1} \xi / \ell)} \\
				&\qquad \leq C C_{E, q} \ell^{|\alpha|} M_{\alpha} \norm{\gamma}_{\mathcal{S}^{M, \ell / 2}_{A, 1, q}} \norm{\Phi}_{C_{e^{\omega_{M}(q \cdot)}}(E)} \frac{e^{\omega_{M}(4 \pi \xi / \ell)}}{e^{\omega_{M}(4 \pi H^{d + 1} \xi / \ell)}} \in L^{1}
			\end{align*}
		We infer from Lemma \ref{l:SuffCondExistencePettisIntegral} and \eqref{eq:LinearMapPettisIntegral} that $\partial^{\alpha} V^{*}_{\gamma} \Phi = \int_{\R^{d}} \partial^{\alpha} M_{\xi} [ \Phi(\xi) * \gamma ] d\xi$ exists as a Pettis integral in $E$ for any $\alpha \in \N^{d}$ and moreover that
			\[ \| \int_{\R^{d}} M_{\xi} [ \Phi(\xi) * \gamma ] d\xi \|_{\D^{M, \ell}_{E}} \leq C' \norm{\gamma}_{\mathcal{S}^{M, \ell / 2}_{A, 1, q}} \norm{\Phi}_{C_{e^{\omega(q \cdot)}}(E)} ,  \]
		for some $C' > 0$. In particular $V^{*}_{\psi} \Phi \in \D^{[M]}_{E}$ and we may conclude the continuity of the map $V^{*}$ by \eqref{eq:GSasL1Sp}. 
	\end{proof}
	
We end this section with three corollaries of Proposition \ref{p:STFTDE}. The first is a density result in case $E$ is a TMIB.
	
	\begin{corollary}
		\label{c:GSDenseDE}
		If $E$ is a TMIB such that \eqref{eq:GrowthCondModulation} holds, then, $\mathcal{S}^{[M]}_{[A]}$ is dense in $\D^{[M]}_{E}$. 
	\end{corollary}
	
	\begin{proof}
		Take any $f \in \D^{[M]}_{E}$ and let $\psi \in \mathcal{S}^{(M)}_{(A)}$ be as in Lemma \ref{l:AdequateWindow}. Then $V_{\psi} f \in C_{[M]}(E)$ by Proposition \ref{p:STFTDE}. Using Lemma \ref{l:DensityCMSp1} (recall that $\mathcal{S}^{\{M\}}_{\{A\}}$ is compactly regular), we find a net $(\Phi_{\tau})_{\tau}$ in $C_{[M]}(\mathcal{S}^{[M]}_{[A]})$ such that $\lim_{\tau} \Phi_{\tau} = V_{\psi} f$. Applying Proposition \ref{p:STFTDE} once more and using \eqref{eq:STFTReconstructTempUltra} we see that $\lim_{\tau} V^{*}_{\psi} \Phi_{\tau} = f$ in $\D^{[M]}_{E}$. Finally, by Proposition \ref{p:STFTGS} we have that $(V^{*}_{\psi} \Phi_{\tau})_{\tau}$ is a net in $\mathcal{S}^{[M]}_{[A]}$, from which the result follows. 
	\end{proof}
	
Second, we are able to easily determine the topological structure of $\D^{[M]}_{E}$.

	\begin{corollary}
		\label{c:TopDE}
		Suppose $E$ is a TMIB or DTMIB such that \eqref{eq:GrowthCondModulation} holds. Then, $\D^{(M)}_{E}$ is a quasinormable and thus distinguished Fr\'{e}chet space, and $\D^{\{M\}}_{E}$ is a complete $(LB)$-space.
	\end{corollary}
	
	\begin{proof}
		By \eqref{eq:STFTReconstructTempUltra} and using Proposition \ref{p:STFTDE} with the window $\psi$ as in Lemma \ref{l:AdequateWindow}, we see that $\D^{[M]}_{E}$ is isomorphic to a complemented subspace of $C_{[M]}(E)$. The result now follows immediately from Lemma \ref{l:TopCMSp}.
	\end{proof}
	
Finally, we are also able to detect when an element of $E$ belongs to the space $\D^{[M]}_{E}$.

	\begin{corollary}
		\label{c:DetectDE}
		Suppose $E$ is a TMIB or DTMIB such that \eqref{eq:GrowthCondModulation} holds. Take $\psi_{0}, \psi_{1} \in \mathcal{S}^{(M)}_{(A)}$ and suppose $\psi = \psi_{0} * \psi_{1} \neq 0$. Let $f \in E$, then, $f \in \D^{[M]}_{E}$ if and only if
			\begin{equation}
				\label{eq:STFTDECond}
				\forall q > 0 ~ (\exists q > 0) : \quad \sup_{\xi \in \R^{d}} e^{\omega_{M}(q \xi)} \norm{V_{\psi} f(\xi)}_{E} < \infty . 
			\end{equation}
	\end{corollary}
	
	\begin{proof}
		If $f \in \D^{[M]}_{E}$, then \eqref{eq:STFTDECond} follows immediately from Proposition \ref{p:STFTDE}. Conversely, suppose $f \in E$ and \eqref{eq:STFTDECond} holds. In particular, we have that $V_{\psi} f(\xi) = (M_{-\xi} f) * \check{\overline{\psi}} \in E$ for any $\xi \in \R^{d}$ by Proposition \ref{p:ExtensionOfConvolutionToWTMIB} and \eqref{eq:GSasL1Sp}. Consequently, by the assumption \eqref{eq:STFTDECond}, it follows from Proposition \ref{p:STFTDE} that $\frac{1}{(\psi, \psi)_{L^{2}}} V^{*}_{\psi} \circ V_{\psi} f \in \D^{[M]}_{E}$. As $f \in E \subseteq \mathcal{S}^{\prime [M]}_{[A]}$, and since $\psi \neq 0$ so that in particular $(\psi, \psi)_{L^{2}} \neq 0$, it follows from \eqref{eq:STFTReconstructTempUltra} that $f = \frac{1}{(\psi, \psi)_{L^{2}}} V^{*}_{\psi} \circ V_{\psi} f \in \D^{[M]}_{E}$. 
	\end{proof}

\section{The convolutor spaces $\OC^{\prime}(\mathcal{S}^{[M]}_{[A]}, E)$}
\label{sec:ConvSp}

For any TMIB or DTMIB $E$ of class $[M]-[A]$, we define the convolutor space
	\[ \OC^{\prime}(\mathcal{S}^{[M]}_{[A]}, E) = \{ f \in \mathcal{S}^{\prime [M]}_{[A]} \mid f * \varphi \in E \text{ for any } \varphi \in \mathcal{S}^{[M]}_{[A]} \} . \]
Fix a $f \in \OC^{\prime}(\mathcal{S}^{[M]}_{[A]}, E)$. By the separate continuity of the convolution mapping \eqref{eq:ConvTempUltraDist} and De Wilde's closed graph theorem, it follows that the mapping $*_{f} : \mathcal{S}^{[M]}_{[A]} \rightarrow E : \varphi \mapsto f * \varphi$ is continuous. We then endow $\OC^{\prime}(\mathcal{S}^{[M]}_{[A]}, E)$ with the topology induced by the embedding
	\[ * : \OC^{\prime}(\mathcal{S}^{[M]}_{[A]}, E) \rightarrow L_{b}(\mathcal{S}^{[M]}_{[A]}, E) : \quad f \mapsto *_{f} , \]
where for any two lcHs $X$ and $Y$, $L_{b}(X, Y)$ denotes the space of all continuous linear mappings $X \rightarrow Y$ with the strong topology. 
	
The aim of this section is to study the topology and structure of the space $\OC^{\prime}(\mathcal{S}^{[M]}_{[A]}, E)$, with our main result being the proof of Theorem \ref{t:StructThm}. Once more our primary tool will be the application of the (adjoint) STFT. Throughout this section we assume $M$ and $A$ to be two weight sequences satisfying $(M.1)$ and $(M.2)'$, and $E$ to be a TMIB or DTMIB of class $[M] - [A]$ satisfying the condition \eqref{eq:GrowthCondModulation}.
	
\subsection{Characterization via the STFT}
	
We consider the continuity of the (adjoint) STFT on $\OC^{\prime}(\mathcal{S}^{[M]}_{[A]}, E)$ with respect to the space $C_{[M], d}(E)$. As a preliminary result, we first verify the bornologicity of the vector-valued tempered ultradistributions.

	\begin{lemma}
		\label{l:TopStructVVTempUltraDist}
		$L_{b}(\mathcal{S}^{[M]}_{[A]}, E)$ is bornological.
	\end{lemma}
	
	\begin{proof}
		In the Beurling case, in view of \cite[Observation 9(a)]{B-D-ProblemTopGrothendieckClassFrechetTSp}, it suffices to show that $\mathcal{S}^{(M)}_{(A)}$ is a quasinormable Fr\'{e}chet T-space. Using \cite[Proposition 4(a)]{B-D-ProblemTopGrothendieckClassFrechetTSp} and the fact that any Fr\'{e}chet-Schwartz space is quasinormable, this is a direct consequence of $\mathcal{S}^{(M)}_{(A)}$ being nuclear and having a continuous norm as shown in \cite[Theorem 1.2]{D-N-V-NuclGSKerThm}. In the Roumieu case, it follows from the nuclearity of $\mathcal{S}^{\{M\}}_{\{A\}}$ that $L_{b}(\mathcal{S}^{\{M\}}_{\{A\}}, E) \cong \mathcal{S}^{\prime \{M\}}_{\{A\}} \widehat{\otimes} E$, where the latter space is a Fr\'{e}chet space \cite[Chapitre II, Th\'{e}or\`{e}me 12, p.~76]{G-ProdTensTopEspNucl} (recall that $\mathcal{S}^{\prime \{M\}}_{\{A\}}$ is a $(FN)$-space), so that in particular $L_{b}(\mathcal{S}^{\{M\}}_{\{A\}}, E)$ is bornological.
	\end{proof}
	
We may now describe the mapping properties of the (adjoint) STFT on $\OC^{\prime}(\mathcal{S}^{[M]}_{[A]}, E)$ as follows.
	
	\begin{proposition}
		\label{p:STFTConvSp}
		Let $\psi, \gamma_{0}, \gamma_{1} \in \mathcal{S}^{(M)}_{(A)}$ and write $\gamma = \gamma_{0} * \gamma_{1}$. Then, the mappings
			\[ V_{\psi} : \OC^{\prime}(\mathcal{S}^{[M]}_{[A]}, E) \rightarrow C_{[M], d}(E) \quad \text{ and } \quad V^{*}_{\gamma} : C_{[M], d}(E) \rightarrow \OC^{\prime}(\mathcal{S}^{[M]}_{[A]}, E) \]
		are well-defined and continuous.
	\end{proposition}
	
	\begin{proof}
		We first show that $V_{\psi}$ is well-defined and continuous. Note that for any $\psi \in \mathcal{S}^{[M]}_{[A]}$ and any $f \in \mathcal{S}^{\prime [M]}_{[A]}$ we have $V_{\psi} f(\xi) = (M_{-\xi} f) * \check{\overline{\psi}} = M_{-\xi}[f * (M_{\xi} \check{\overline{\psi}} )]$. As a result, it suffices to show that $f \mapsto f * (M_{\xi} \check{\overline{\psi}} )$ is a continuous linear map $\OC^{\prime}(\mathcal{S}^{[M]}_{[A]}, E) \rightarrow C_{[M], d}(E)$ and $\Phi \mapsto M_{-\xi} \Phi$ is a continuous linear map $C_{[M], d}(E) \rightarrow C_{[M], d}(E)$. We start with the first mapping. Let $B \subset L_{b}(\mathcal{S}^{[M]}_{[A]}, E)$ be a bounded set, then $B$ is equicontinuous by the Banach-Steinhaus theorem. Let $X_{[\ell]}$ be as in \eqref{eq:XSpaces}, then it follows that for some $\ell > 0$ (resp. for any $\ell > 0$) $B$ may be extended to a bounded subset of $L_{b}(X_{[\ell]}, E)$. For any $\xi \in \R^{d}$ and $L \in B$ we now have,
			\[ \| L(M_{\xi} \check{\overline{\psi}}) \|_{E} \leq \| L \|_{L_{b}(X_{[\ell]}, E)} \| M_{\xi} \check{\overline{\psi}} \|_{X_{[\ell]}} \leq C \| \check{\overline{\psi}} \|_{X_{[\ell / 2]}} \| L \|_{L_{b}(X_{[\ell]}, E)} e^{\omega_{M}(4 \pi \xi / \ell)} .  \]
		This shows that the linear map
			\[ L_{b}(\mathcal{S}^{[M]}_{[A]}, E) \rightarrow C_{[M], d}(E) : \quad L \mapsto (\xi \mapsto L(M_{\xi} \check{\overline{\psi}}) ) , \]
		is well-defined and moreover it maps bounded sets of $L_{b}(\mathcal{S}^{[M]}_{[A]}, E)$ into bounded sets of $C_{[M], d}(E)$. As $L_{b}(\mathcal{S}^{[M]}_{[A]}, E)$ is bornological by Lemma \ref{l:TopStructVVTempUltraDist}, it follows that the mapping above is continuous. Now, as the first map is exactly $f \mapsto *_{f}(M_{\xi} \check{\overline{\psi}})$, its continuity follows directly. Moving on, if $q_{0}, q_{1} > 0$ are such that $\nu_{E}(\xi) e^{\omega_{M}(q_{0} \xi)} \leq C e^{\omega_{M}(q_{1} \xi)}$ for some $C > 0$. Then, for any $\Phi \in C_{e^{-\omega_{M}(q_{0} \cdot)}}(E)$,
			\[ e^{-\omega_{M}(q_{1} \xi)} \norm{M_{- \xi} \Phi(\xi)}_{E} \leq \nu_{E}(\xi) e^{-\omega_{M}(q_{1} \xi)} \norm{\Phi(\xi)}_{E} \leq C \norm{\Phi}_{C_{e^{-\omega_{M}(q_{0} \cdot)}}(E)} . \]
		The continuity of the second map then follows by \eqref{eq:GrowthCondModulation}.
		
		Next, we consider $V^{*}_{\gamma}$. As $C_{[M], d}(E)$ is continuously embedded into $C_{[M], d}(\mathcal{S}^{\prime [M]}_{[A]})$, it follows from Proposition \ref{p:STFTTempUltra} that $V^{*}_{\gamma} \Phi \in \mathcal{S}^{\prime [M]}_{[A]}$ for any $\Phi \in C_{[M], d}(E)$. We now show that $V^{*}_{\gamma} \Phi \in \OC^{\prime}(\mathcal{S}^{[M]}_{[A]}, E)$. Take any $\varphi \in \mathcal{S}^{[M]}_{[A]}$. Then, by \eqref{eq:LinearMapPettisIntegral},
			\begin{align*}
				V^{*}_{\gamma} \Phi * \varphi &= \int_{\R^{d}} (M_{\xi} [\Phi(\xi) * \gamma]) * \varphi d\xi = \int_{\R^{d}} M_{\xi} [ \Phi(\xi) * \gamma * (M_{-\xi} \varphi) ] d\xi \\
				&= \int_{\R^{d}} M_{\xi} [\Phi(\xi) * V_{\check{\overline{\gamma}}} \varphi(\xi)] d\xi . 
			\end{align*}
		Lemma \ref{l:InducedBilinearMap} applied to the continuous bilinear map $* : E \times \mathcal{S}^{[M]}_{[A]} \rightarrow E$ and the family $\Gamma_{\xi} = M_{\xi}$, gives by \eqref{eq:GrowthCondModulation} the hypocontinuous bilinear map
			\[ C_{[M]}(\mathcal{S}^{[M]}_{[A]}) \times C_{[M], d}(E) \rightarrow E : \quad (\Phi_{\mathcal{S}^{[M]}_{[A]}}, \Phi_{E}) \mapsto \int_{\R^{d}} M_{\xi} [\Phi_{\mathcal{S}^{[M]}_{[A]}}(\xi) * \Phi_{E}(-\xi)] d\xi . \]
		In particular, by Proposition \ref{p:STFTGS}, we find that $V^{*}_{\gamma} \Phi * \varphi \in E$ for any $\varphi \in \mathcal{S}^{[M]}_{[A]}$, so that $V^{*}_{\gamma} \Phi \in \OC^{\prime}(\mathcal{S}^{[M]}_{[A]}, E)$. Moreover, if $B \subset C_{[M], d}(E)$ is a bounded subset, then $\{ *_{V^{*}_{\gamma} \Phi} : \Phi \in B \}$ will be an equicontinuous subset of $L_{b}(\mathcal{S}^{[M]}_{[A]}, E)$, so that by the Banach-Steinhaus Theorem the set $\{ V^{*}_{\gamma} \Phi : \Phi \in B \}$ will be bounded in $\OC^{\prime}(\mathcal{S}^{[M]}_{[A]}, E)$. As $C_{[M], d}(E)$ is an $(LB)$-space (a Fr\'{e}chet space) by Lemma \ref{l:TopCMSp}, it is bornological. The continuity of $V^{*}_{\gamma}$ now follows.
	\end{proof}

	\begin{corollary}
		\label{c:denseinclusions} 
		The following dense inclusions hold:
			\[ \D^{[M]}_{E} \hookrightarrow E \hookrightarrow \OC^{\prime}(\mathcal{S}^{[M]}_{[A]}, E)  . \]
		If $E$ is a TMIB, then, $\mathcal{S}^{[M]}_{[A]}$ is dense in $\OC^{\prime}(\mathcal{S}^{[M]}_{[A]}, E)$. 
	\end{corollary}
	
	\begin{proof}
		As $E$ is continuously contained in $\OC^{\prime}(\mathcal{S}^{[M]}_{[A]}, E)$ by \eqref{eq:ConvWTMIB}, it suffices by Corollary \ref{c:DEdenseinE} to show that $\D^{[M]}_{E}$ is dense in $\OC^{\prime}(\mathcal{S}^{[M]}_{[A]}, E)$. Let $\psi \in \mathcal{S}^{(M)}_{(A)}$ be as in Lemma \ref{l:AdequateWindow}. For any $f \in \OC^{\prime}(\mathcal{S}^{[M]}_{[A]}, E)$, we have $V_{\psi} f \in C_{[M], d}(E)$ by Proposition \ref{p:STFTConvSp}. Using Lemma \ref{l:DensityCMSp2}, we find a net $(\Phi_{\tau})_{\tau}$ in $C_{[M]}(E)$ such that $\lim_{\tau} \Phi_{\tau} = V_{\psi} f$. By \eqref{eq:STFTReconstructTempUltra} and another application of Proposition \ref{p:STFTConvSp} it follows that $\lim_{\tau} V^{*}_{\psi} \Phi_{\tau} = f$. Proposition \ref{p:STFTDE} tells us that $(V^{*}_{\psi} \Phi_{\tau})_{\tau}$ is a net in $\D^{[M]}_{E}$, which shows the required density. The final statement now follows directly.
	\end{proof}
	
	\begin{corollary}
		\label{c:TopConvSp}
		$\OC^{\prime}(\mathcal{S}^{(M)}_{(A)}, E)$ is a complete $(LB)$-space, and $\OC^{\prime}(\mathcal{S}^{\{M\}}_{\{A\}}, E)$ is a quasinormable and thus distinguished Fr\'{e}chet space. 
	\end{corollary}
	
	\begin{proof}
		By Proposition \ref{p:STFTConvSp} and \eqref{eq:STFTReconstructTempUltra} we see that $\OC^{\prime}(\mathcal{S}^{[M]}_{[A]}, E)$ is isomorphic to a complemented subspace of $C_{[M], d}(E)$, hence the result follows immediately from Lemma \ref{l:TopCMSp}.
	\end{proof}
	
In the non-quasianalytic case, we get the following.
	
	\begin{theorem}
		Let $f \in \D^{\prime [M]}$. Then, $f \in \OC^{\prime}(\mathcal{S}^{[M]}_{[A]}, E)$ if and only if $f * \varphi \in E$ for any $\varphi \in \D^{[M]}$. 
	\end{theorem}
	
	\begin{proof}
		If $f \in \OC^{\prime}(\mathcal{S}^{[M]}_{[A]}, E)$ then trivially $f * \varphi \in E$ for any $\varphi \in \D^{[M]} \subseteq \mathcal{S}^{[M]}_{[A]}$. We now show the reverse implication. Take any $\psi \in \D^{[M]} \setminus \{ 0 \}$. By an analogous argument as in the proof of continuity of $V_{\psi}$ in Proposition \ref{p:STFTConvSp}, one can show that $V_{\psi} f \in  C_{[M], d}(E)$. Consequently, by Proposition \ref{p:STFTConvSp}, $\frac{1}{(\psi, \psi)_{L^{2}}} V^{*}_{\psi} \circ V_{\psi} f \in \OC^{\prime}(\mathcal{S}^{[M]}_{[A]}, E)$. As one can similarly show that \eqref{eq:STFTReconstructTempUltra} holds in $\D^{\prime [M]}$ (with windows in $\D^{[M]}$), it follows that $f \in \OC^{\prime}(\mathcal{S}^{[M]}_{[A]}, E)$.  
	\end{proof}

\subsection{Structural theorem}

We now move on to prove Theorem \ref{t:StructThm}. To do this, we introduce the following spaces of vector-valued multi-sequences. Let $E$ be a Banach space. For any $\ell > 0$ we define $\Lambda_{M, \ell}(E)$ as the Banach space of all (multi-indexed) sequences $(e_{\alpha})_{\alpha \in \N^{d}} \in E^{\N^{d}}$ such that
	\[ \norm{(e_{\alpha})_{\alpha \in \N^{d}}}_{\Lambda_{M, \ell}} = \sup_{\alpha \in \N^{d}} \ell^{|\alpha|} M_{\alpha} \norm{e_{\alpha}}_{E} < \infty . \]
Then, we define the spaces,
	\[ \Lambda_{(M)}(E) = \varinjlim_{\ell \rightarrow 0^{+}} \Lambda_{M, \ell}(E) , \qquad \Lambda_{\{M\}}(E) = \varprojlim_{\ell \rightarrow \infty} \Lambda_{M, \ell}(E) . \]
$\Lambda_{(M)}(E)$ is a complete $(LB)$-space by \cite[Theorem 2.6]{B-M-S-ProjDescrWeighIndLim} and $\Lambda_{\{M\}}(E)$ is a Fr\'{e}chet space. For a Banach space $F$, we write $\Lambda^{\prime}_{(M)}(F) := \Lambda_{\{1/M\}}(F)$ and $\Lambda^{\prime}_{\{M\}}(F) := \Lambda_{(1/M)}(F)$. We then have the following canonical isomorphisms of lcHs
	\[ (\Lambda_{[M]}(E))^{\prime} \cong \Lambda^{\prime}_{[M]}(E^{\prime}) .  \]

We fix a TMIB $E$ of class $[M] - [A]$ satisfying \eqref{eq:GrowthCondModulation}. Note that by Remark \ref{r:dualiswTMIB}, $E^{\prime}$ will be a DTMIB that satisfies \eqref{eq:GrowthCondModulation} as well. We can now reformulate Theorem \ref{t:StructThm} as follows.

	\begin{theorem}
		\label{t:StructMapping}
		For any $f \in \mathcal{S}^{\prime [M]}_{[A]}$, the following statements are equivalent:
			\begin{itemize}
				\item[$(i)$] $f * \varphi \in E$ for any $\varphi \in \mathcal{S}^{[M]}_{[A]}$;
				\item[$(ii)$] there exists $(f_{\alpha})_{\alpha \in \N^{d}} \in \Lambda_{[M]}(E)$ such that $f = \sum_{\alpha \in \N^{d}} f_{\alpha}^{(\alpha)}$.
			\end{itemize}
	\end{theorem}

For the remainder of this section, we focus on proving Theorem \ref{t:StructMapping}. To do this, we introduce the mapping
	\begin{equation} 
		\label{eq:Smap}
		S : \Lambda_{[M]}(E) \rightarrow \OC^{\prime}(\mathcal{S}^{[M]}_{[A]}, E): \quad (f_{\alpha})_{\alpha \in \N^{d}} \mapsto \sum_{\alpha \in \N^{d}} f^{(\alpha)}_{\alpha} .
	\end{equation}
Showing that $S$ is well-defined implies that $(ii) \Rightarrow (i)$ of Theorem \ref{t:StructMapping} holds, while showing that $S$ is surjective proves $(i) \Rightarrow (ii)$. We start with the former.

	\begin{lemma}
		\label{l:SWelldefined}
		$S$ is a well-defined continuous linear map. 
	\end{lemma}
	
	\begin{proof}
		Consider any $(f_{\alpha})_{\alpha \in \N^{d}} \in \Lambda_{[M]}(E)$ and take an arbitrary bounded subset $B \subset \mathcal{S}^{[M]}_{[A]}$. Then, for certain $q, \ell > 0$ we have $(f_{\alpha})_{\alpha \in \N^{d}} \in \Lambda_{M, \ell}(E)$, $B \subset \mathcal{S}^{M, \ell / 2}_{A, 1, q}$ and is bounded there, and $\omega_{E}(x) \leq C_{E, q} e^{\omega_{A}(q x)}$. We get, by \eqref{eq:convolutionineqE},
			\begin{align*} 
				\norm{f^{(\alpha)}_{\alpha} * \varphi}_{E} &= \norm{f_{\alpha} * \varphi^{(\alpha)}}_{E} \leq \norm{f_{\alpha}}_{E} \norm{\varphi^{(\alpha)}}_{L^{1}_{\omega_{E}}} \leq C_{E, q} \left(\frac{\ell}{2}\right)^{|\alpha|} M_{\alpha} \norm{f_{\alpha}}_{E} \norm{\varphi}_{\mathcal{S}^{M, \ell / 2}_{A, 1, q}} \\
				&\leq \frac{C_{E, q} \norm{(f_{\alpha})_{\alpha \in \N^{d}}}_{\Lambda_{M, \ell}} \norm{\varphi}_{\mathcal{S}^{M, \ell / 2}_{A, 1, q}}}{2^{|\alpha|}} , 
			\end{align*}
		so that $(\sum_{\alpha \in \N^{d}} f^{(\alpha)}_{\alpha}) * \varphi = \sum_{\alpha \in \N^{d}} f^{(\alpha)}_{\alpha} * \varphi \in E$ for any $\varphi \in B$, and moreover,
			\[ \sup_{\varphi \in B} \|  (\sum_{\alpha \in \N^{d}} f^{(\alpha)}_{\alpha}) * \varphi \|_{E} \leq (2^{d} C_{E, q} \sup_{\varphi \in B} \norm{\varphi}_{\mathcal{S}^{M, \ell / 2}_{A, 1, q}}) \norm{(f_{\alpha})_{\alpha \in \N^{d}}}_{\Lambda_{M, \ell}(E)} , \]
		from which the result follows. 
	\end{proof}
	
To show that the mapping $S$ is surjective, we will employ the following abstract criterion. 

	\begin{lemma}
		\label{l:AbstractSurjCrit}
		Let $X, Y,$ and $Z$ be lcHs and $S : X \rightarrow Y$ be a continuous linear mapping. Suppose $X$ is Mackey, $X / \ker S$ is complete, $\im S$ is Mackey for the topology induced by $Y$, and there exists a topological isomorphism $\iota : Z \rightarrow Y^{\prime}$. We write $R = S^{t} \circ \iota$. Then, $S$ is surjective if the following two conditions are satisfied:
			\begin{itemize}
				\item[(1)] $R$ is injective;
				\item[(2)] $\im R$ is weakly closed in $X^{\prime}$.
			\end{itemize}
	\end{lemma}
	
	\begin{proof}
		Note that if $R$ is injective and has weakly closed image, then so does $S^{t}$. The result then follows from \cite[Lemma 1]{D-N-V-SpUltraDistVanishInf}.
	\end{proof}
	
To apply Lemma \ref{l:AbstractSurjCrit}, a good description of the dual of $\OC^{\prime}(\mathcal{S}^{[M]}_{[A]}, E)$ is required. In the next result we show this to be $\D^{[M]}_{E^{\prime}}$. 

	\begin{proposition}
		\label{p:dualConvSp}
		Let $\psi, \gamma, \gamma_{0}, \gamma_{1} \in \mathcal{S}^{(M)}_{(A)}$ be such that $\gamma = \gamma_{0} * \gamma_{1}$ and $(\gamma, \psi)_{L^{2}} \neq 0$. The mapping
			\[ \iota : \D^{[M]}_{E^{\prime}} \rightarrow (\OC^{\prime}(\mathcal{S}^{[M]}_{[A]}, E))^{\prime}_{b} : \quad g \mapsto (f \mapsto \frac{1}{(\gamma, \psi)_{L^{2}}} \int_{\R^{d}} \ev{V_{\psi} g(\xi)}{V_{\overline{\gamma}} f(- \xi)} d\xi) \]
		is a topological isomorphism. 
	\end{proposition}
	
	\begin{proof}
		As $\D^{[M]}_{E^{\prime}}$ is a Fr\'{e}chet space (an $(LB)$-space) it is webbed and $(\OC^{\prime}(\mathcal{S}^{[M]}_{[A]}, E))^{\prime}_{b}$ is ultrabornological by Corollary \ref{c:TopConvSp}. Consequently, it suffices by De Wilde's open mapping theorem to show that $\iota$ is a bijective continuous linear map. 
		
		Using Lemma \ref{l:InducedBilinearMap} with $A = \ev{\cdot}{\cdot}$ and $\Gamma_{\xi} \equiv \id$, we find the hypocontinuous bilinear map
			\[ C_{[M]}(E^{\prime}) \times C_{[M], d}(E) \rightarrow \C : \quad (\Phi_{E^{\prime}}, \Phi_{E}) \mapsto \int_{\R^{d}} \ev{\Phi_{E^{\prime}}(\xi)}{\Phi_{E}(-\xi)} d\xi . \]
		By Proposition \ref{p:STFTDE} and Proposition \ref{p:STFTConvSp}, we see that for any bounded set $B \subset \D^{[M]}_{E^{\prime}}$, the family of linear maps 
			\[ \OC^{\prime}(\mathcal{S}^{[M]}_{[A]}, E) \rightarrow \C : \quad f \mapsto \frac{1}{(\gamma, \psi)_{L^{2}}} \int_{\R^{d}} \ev{V_{\psi} g(\xi)}{V_{\overline{\gamma}} f(- \xi)} d\xi , \qquad g \in B , \]
		form an equicontinuous, and hence bounded subset of $(\OC^{\prime}(\mathcal{S}^{[M]}_{[A]}, E))^{\prime}_{b}$. As $\D^{[M]}_{E^{\prime}}$ is a Fr\'{e}chet space (an $(LB)$-space), it is bornological. Whence we find that $\iota$ is continuous.
		
		We now show the injectivity of $\iota$. Suppose that for $g_{1}, g_{2} \in \D^{[M]}_{E^{\prime}}$ we have $\iota(g_{1}) = \iota(g_{2})$. In particular,
			\[ \frac{1}{(\gamma, \psi)_{L^{2}}} \int_{\R^{d}} \ev{V_{\psi} g_{1}(\xi)}{V_{\overline{\gamma}} \varphi(- \xi)} d\xi = \frac{1}{(\gamma, \psi)_{L^{2}}} \int_{\R^{d}} \ev{V_{\psi} g_{2}(\xi)}{V_{\overline{\gamma}} \varphi(- \xi)} d\xi , \qquad \forall \varphi \in \mathcal{S}^{[M]}_{[A]} . \]
		By \eqref{eq:STFTDesingularizationTempUltra} it follows that $g_{1}$ and $g_{2}$ coincide on $\mathcal{S}^{[M]}_{[A]}$. As $\mathcal{S}^{[M]}_{[A]}$ is dense in $E$, we see that $g_{1} = g_{2}$ in $\D^{[M]}_{E^{\prime}}$. 
		
		Finally, we show that $\iota$ is surjective. Take any $g \in (\OC^{\prime}(\mathcal{S}^{[M]}_{[A]}, E))^{\prime}_{b}$. By Corollary \ref{c:denseinclusions} it follows that $g$ is an element of $E^{\prime}$, whence so is $V_{\psi} g(\xi) = (M_{-\xi} g) * \check{\overline{\psi}}$ for any $\xi \in \R^{d}$. As $g : \OC^{\prime}(\mathcal{S}^{[M]}_{[A]}, E) \rightarrow \C$ is a continuous map, there exists a bounded subset $B \subset \mathcal{S}^{[M]}_{[A]}$ such that
			\begin{align*}
				\| V_{\psi} g (\xi) \|_{E^{\prime}} &= \sup_{\norm{e}_{E} \leq 1} |\ev{(M_{-\xi} g) * \check{\overline{\psi}}}{e}| = \sup_{\norm{e}_{E} \leq 1} |\ev{g}{M_{-\xi} (e * \overline{\psi})}| \\
				&\leq \sup_{\norm{e} \leq 1} \sup_{\gamma \in B} \| [M_{-\xi}(e * \overline{\psi})] * \gamma \|_{E} = \sup_{\norm{e}_{E} \leq 1} \sup_{\gamma \in B} \| V_{\check{\overline{\gamma}}}[e * \overline{\psi}](\xi) \|_{E} . 
			\end{align*}
		By Lemma \ref{l:ConvolutionEtoDE} and Proposition \ref{p:STFTDE} we have that $\{ V_{\check{\overline{\gamma}}}[e * \overline{\psi}](\xi) : \norm{e}_{E} \leq 1, \gamma \in B \}$ is a bounded set in $C_{[M]}(E)$. Hence, for every $q > 0$ (for some $q > 0$) we have that
			\[ \sup_{\xi \in \R^{d}} e^{\omega_{M}(q \xi)} \| V_{\psi} g(\xi) \|_{E^{\prime}} \leq \sup_{\xi \in \R^{d}} \sup_{\norm{e}_{E} \leq 1} \sup_{\gamma \in B} e^{\omega_{M}(q \xi)} \| V_{\check{\overline{\gamma}}}[e * \overline{\psi}](\xi) \|_{E}  < \infty . \]
		We then infer from Corollary \ref{c:DetectDE} that $g \in \D^{[M]}_{E^{\prime}}$. By \eqref{eq:STFTDesingularizationTempUltra} we have that $g$ and $\iota(g)$ coincide on $\mathcal{S}^{[M]}_{[A]}$, so that by Corollary \ref{c:denseinclusions} they describe the same element in $(\OC^{\prime}(\mathcal{S}^{[M]}_{[A]}, E))^{\prime}_{b}$. 
	\end{proof}
	
Let $\psi \in \mathcal{S}^{(M)}_{(A)}$ be as in Lemma \ref{l:AdequateWindow} and take $\iota$ with respect to $\psi$ according to Proposition \ref{p:dualConvSp}. We now consider the map $R = S^{t} \circ \iota$. By the previous result, we know exactly what this mapping looks like.

	\begin{corollary}
		The map $R$ is the continuous linear mapping
			\[ \D^{[M]}_{E^{\prime}} \rightarrow \Lambda^{\prime}_{[M]}(E^{\prime}) : \quad g \mapsto ((-1)^{|\alpha|} g^{(\alpha)})_{\alpha \in \N^{d}} . \]
		In particular, $R$ is injective.
	\end{corollary}
	
	\begin{proof}
		The mapping above is clearly a well-defined continuous linear map. Now, for any $g \in \D^{[M]}_{E^{\prime}}$ and $(f_{\alpha})_{\alpha \in \N^{d}} \in \Lambda_{[M]}(E)$, we have that
			\[ \ev{R(g)}{(f_{\alpha})_{\alpha \in \N^{d}}} = \sum_{\alpha \in \N^{d}} \int_{\R^{d}} \ev{V_{\psi} g(\xi)}{V_{\overline{\psi}}[f^{(\alpha)}_{\alpha}](-\xi)} d\xi . \]
		By \eqref{eq:STFTDesingularizationTempUltra} it holds that for any $\varphi \in \mathcal{S}^{[M]}_{[A]}$,
			\[ \int_{\R^{d}} \ev{V_{\psi} g(\xi)}{V_{\overline{\psi}}[\varphi^{(\alpha)}](-\xi)} d\xi = \ev{g}{\varphi^{(\alpha)}} = (-1)^{|\alpha|} \ev{g^{(\alpha)}}{\varphi} . \]
		As $\mathcal{S}^{[M]}_{[A]}$ is dense in $E$, we may extend the previous identity, so that in particular,
			\[ \ev{R(g)}{(f_{\alpha})_{\alpha \in \N^{d}}} = \sum_{\alpha \in \N^{d}} (-1)^{|\alpha|} \ev{g^{(\alpha)}}{f_{\alpha}} . \]
		Consequently, the mappings coincide. 
	\end{proof}
	
We are now sufficiently prepared to prove Theorem \ref{t:StructMapping}.
	
	\begin{proof}[Proof of Theorem \ref{t:StructMapping}]
		$(i) \Rightarrow (ii)$. It suffices to show that $S$ is surjective, which we do by applying Lemma \ref{l:AbstractSurjCrit} to $X = \Lambda_{[M]}(E)$, $Y = \OC^{\prime}(\mathcal{S}^{[M]}_{[A]}, E)$ and $Z = \D^{[M]}_{E^{\prime}}$. Now, $\Lambda_{[M]}(E)$ is Mackey as it is barreled, while $\Lambda_{[M]}(E) / \ker S$ is complete as $\Lambda_{[M]}(E)$ is complete. Next, we show that $\im S$ is Mackey. As $\OC^{\prime}(\mathcal{S}^{\{M\}}_{\{A\}}, E)$ is a Fr\'{e}chet space, no further work is needed in the Roumieu case. For the Beurling case, we will show that $\im S$ is Mackey by showing that it is infrabarreled, that is, we have to prove that every strongly bounded set $B$ in $(\im S)^{\prime}$ is equicontinuous. As $\im S$ is dense in $\OC^{\prime}(\mathcal{S}^{(M)}_{(A)}, E)$ (since $R$ is injective, hence also $S^{t}$), we have by Proposition \ref{p:dualConvSp} that $\iota(\D^{[M]}_{E^{\prime}}) = (\im S)^{\prime}$. For arbitrary $\ell > 0$, consider the set
			\[ V_{\ell} = \left\{ \frac{\partial^{\alpha} e}{\ell^{|\alpha|} M_{\alpha}} : \alpha \in \N^{d}, \norm{e}_{E} \leq 1 \right\} \subseteq \im S . \]
		As $S$ is continuous, $V_{\ell}$ is a bounded set in $\im S$, consequently $\sup_{h \in B} \sup_{f \in V_{\ell}} |\ev{h}{f}| < \infty$. Now, by \eqref{eq:STFTDesingularizationTempUltra},
			\begin{align*}
				\sup_{h \in B} \sup_{f \in V_{\ell}} |\ev{h}{f}| &= \sup_{h \in B} \sup_{\alpha \in \N^{d}} \sup_{\norm{e}_{E} \leq 1} \left| \ev{h}{\frac{\partial^{\alpha} e}{\ell^{|\alpha|} M_{\alpha}}} \right| \\
				&= \sup_{g \in \iota^{-1}(B)} \sup_{\alpha \in \N^{d}} \sup_{\norm{e}_{E} \leq 1} \left| \int_{\R^{d}} \ev{V_{\psi} g(\xi)}{V_{\overline{\psi}}\left[\frac{\partial^{\alpha} e}{\ell^{|\alpha|} M_{\alpha}}\right](-\xi)} d\xi \right| \\
				&= \sup_{g \in \iota^{-1}(B)} \sup_{\alpha \in \N^{d}} \frac{1}{\ell^{|\alpha|} M_{\alpha}} \sup_{\norm{e}_{E} \leq 1} |\ev{g^{(\alpha)}}{e}| \\
				&= \sup_{g \in \iota^{-1}(B)} \sup_{\alpha \in \N^{d}} \frac{\norm{g^{(\alpha)}}_{E^{\prime}}}{\ell^{|\alpha|} M_{\alpha}} . 
			\end{align*}
		Hence,
			\[ \sup_{g \in \iota^{-1}(B)} \norm{g}_{\D^{M, \ell}_{E^{\prime}}} < \infty , \qquad \forall \ell > 0 , \]
		which shows that $\iota^{-1}(B)$ is a bounded set in $\D^{(M)}_{E^{\prime}}$. Then, because of Proposition \ref{p:dualConvSp} and the fact that $\OC^{\prime}(\mathcal{S}^{(M)}_{(A)}, E)$ is barreled by Corollary \ref{c:TopConvSp}, $B$ is an equicontinuous subset of $(\im S)^{\prime}$. 
		
		It remains to show that $\im R$ is weakly closed in $\Lambda^{\prime}_{[M]}(E^{\prime})$. Let $(g_{j})_{j}$ be a net in $\D^{[M]}_{E^{\prime}}$ and $(h_{\alpha})_{\alpha \in \N^{d}}$ an element of $\Lambda^{\prime}_{[M]}(E^{\prime})$ such that $R(g_{j}) = ((-1)^{|\alpha|} g_{j}^{(\alpha)})_{\alpha \in \N^{d}} \rightarrow (h_{\alpha})_{\alpha \in \N^{d}}$ weakly in $\Lambda^{\prime}_{[M]}(E^{\prime})$. In particular, $(-1)^{|\alpha|} g^{(\alpha)} \rightarrow h_{\alpha}$ weakly in $E^{\prime}$ for any $\alpha \in \N^{d}$. Note that for any $\varphi \in \mathcal{S}^{[M]}_{[A]}$ and $\alpha \in \N^{d}$,
			\[ \ev{h^{(\alpha)}_{0}}{\varphi} = (-1)^{|\alpha|} \ev{h_{0}}{\varphi^{(\alpha)}} = (-1)^{|\alpha|} \lim_{j} ~ \ev{g_{j}}{\varphi^{(\alpha)}} = \lim_{j} ~ \ev{g_{j}^{(\alpha)}}{\varphi} = (-1)^{|\alpha|} \ev{h_{\alpha}}{\varphi} .  \]
		As $\mathcal{S}^{[M]}_{[A]}$ is dense in $E$, we see that $h_{0}^{(\alpha)} = (-1)^{|\alpha|} h_{\alpha} \in E^{\prime}$ for any $\alpha \in \N^{d}$. Since
			\[ (\| h_{0}^{(\alpha)} \|_{E^{\prime}})_{\alpha \in \N^{d}} = (\norm{h_{\alpha}}_{E^{\prime}})_{\alpha \in \N^{d}} \in \Lambda^{\prime}_{[M]}(\R) , \]
		we see that $h_{0} \in \D^{[M]}_{E^{\prime}}$. Consequently, $(h_{\alpha})_{\alpha \in \N^{d}} = ((-1)^{|\alpha|} h^{(\alpha)}_{0})_{\alpha \in \N^{d}} \in \im R$.	
		
		$(ii) \Rightarrow (i)$. We need to show that $S$ is well-defined. This was done in Lemma \ref{l:SWelldefined}.
	\end{proof}
	
\section{Spaces associated to a DTMIB}
\label{sec:DTMIB}

Let $M$ and $A$ be weight sequences satisfying $(M.1)$ and $(M.2)'$, and suppose $E^{\prime}$ is a DTMIB coming from a TMIB $E$ of class $[M] - [A]$ satisfying \eqref{eq:GrowthCondModulation}. In general one may not apply Theorem \ref{t:StructThm} to obtain the structure of $\OC^{\prime}(\mathcal{S}^{[M]}_{[A]}, E^{\prime})$ as $\mathcal{S}^{[M]}_{[A]}$ is not necessarily dense in $E^{\prime}$, see Remark \ref{r:dualiswTMIB}. In this section we will show that the first structural theorem may still be obtained if one looks at the strong dual of $\D^{[M]}_{E}$, commonly denoted by $\D^{\prime [M]}_{E^{\prime}}$. By Corollary \ref{c:DEdenseinE} and Corollary \ref{c:GSDenseDE} we know that $\D^{\prime [M]}_{E^{\prime}}$ is a subspace of $\mathcal{S}^{\prime [M]}_{[A]}$ containing $E^{\prime}$. However, we will now see that it is exactly the space $\OC^{\prime}(\mathcal{S}^{[M]}_{[A]}, E^{\prime})$.

	\begin{proposition}
		\label{p:dualDEinOC}
		$\D^{\prime [M]}_{E^{\prime}} = \OC^{\prime}(\mathcal{S}^{[M]}_{[A]}, E^{\prime})$ as locally convex spaces.
	\end{proposition}
	
	\begin{proof}
		We first show that $\D^{\prime [M]}_{E^{\prime}}$ is continuously contained in $\OC^{\prime}(\mathcal{S}^{[M]}_{[A]}, E^{\prime})$. For any $f \in \D^{\prime [M]}_{E^{\prime}}$ and $\psi \in \mathcal{S}^{[M]}_{[A]}$, by Lemma \ref{l:ConvolutionEtoDE}, $f * \psi$ may be seen as an element of $E^{\prime}$ via the evaluation $\ev{f * \psi}{e} = \ev{f}{e * \check{\psi}}$, $e \in E$. If $B_{0} \subset \mathcal{S}^{[M]}_{[A]}$ is a bounded set, then $B_{1} = \{ e * \check{\psi} : \psi \in B_{0}, \norm{e}_{E} \leq 1 \}$ is a bounded set in $\D^{[M]}_{E}$ by Lemma \ref{l:ConvolutionEtoDE}. Consequently, it follows that for any bounded set $B \subset \D^{\prime [M]}_{E^{\prime}}$ we have
			\[ \sup_{f \in B} \sup_{\psi \in B_{0}} \norm{f * \psi}_{E^{\prime}} = \sup_{f \in B} \sup_{\varphi \in B_{1}} |\ev{f}{\varphi}| < \infty . \]
		Thus, $B$ is a bounded subset of $\OC^{\prime}(\mathcal{S}^{[M]}_{[A]}, E^{\prime})$. As $\D^{\prime [M]}_{E^{\prime}}$ is bornological by Corollary \ref{c:TopDE}, we find the first continuous inclusion.
		
		To show that $\OC^{\prime}(\mathcal{S}^{[M]}_{[A]}, E^{\prime})$ is continuously contained in $\D^{\prime [M]}_{E^{\prime}}$, it suffices by Proposition \ref{p:STFTConvSp} and \eqref{eq:STFTReconstructTempUltra} to show that for any $\gamma = \gamma_{0} * \gamma_{1}$, with $\gamma_{0}, \gamma_{1} \in \mathcal{S}^{(M)}_{(A)}$, the mapping $V^{*}_{\gamma} : C_{[M], d}(E^{\prime}) \rightarrow \D^{\prime [M]}_{E^{\prime}}$ is well-defined and continuous. For any $f \in E^{\prime}$ and $\varphi \in \D^{[M]}_{E}$ we have that $\ev{M_{\xi}[f * \gamma]}{\varphi} = \ev{f}{V_{\overline{\gamma}} \varphi(-\xi)}$. By Lemma \ref{l:InducedBilinearMap} and Proposition \ref{p:STFTDE} we find the hypocontinuous bilinear mapping
			\[ \D^{[M]}_{E} \times C_{[M], d}(E^{\prime}) \rightarrow \C : \quad (\varphi, \Phi) \mapsto \int_{\R^{d}} \ev{\Phi(\xi)}{V_{\overline{\gamma}} \varphi(-\xi)} d\xi . \]
		The right-hand side is exactly $\ev{V^{*}_{\gamma} \Phi}{\varphi}$, from which we already see that $V^{*}_{\gamma} \Phi \in \D^{\prime [M]}_{E^{\prime}}$. Also, for any bounded set $B \subset C_{[M], d}(E^{\prime})$, the set $\{ V^{*}_{\gamma} \Phi : \Phi \in B \}$ is an equicontinuous and thus bounded subset of $\D^{\prime [M]}_{E^{\prime}}$. As $C_{[M], d}(E^{\prime})$ is bornological, the continuity of $V^{*}_{\gamma} : C_{[M], d}(E^{\prime}) \rightarrow \D^{\prime [M]}_{E^{\prime}}$ follows. 
	\end{proof}
	
	\begin{corollary}
		If $E$ is reflexive, then so are $\D^{[M]}_{E}$ and $\OC^{\prime}(\mathcal{S}^{[M]}_{[A]}, E)$.
	\end{corollary}
	
	\begin{proof}
		Note that if $E$ is a reflexive TMIB satisfying \eqref{eq:GrowthCondModulation}, then $E^{\prime}$ is also a reflexive TMIB satisfying \eqref{eq:GrowthCondModulation}, see Remark \ref{r:dualiswTMIB}. By Proposition \ref{p:dualDEinOC} we have that $\OC^{\prime}(\mathcal{S}^{[M]}_{[A]}, E)$ is exactly the dual of $\D^{[M]}_{E^{\prime}}$, so it suffices to show that $\D^{[M]}_{E}$ is reflexive for any reflexive TMIB $E$. Now, $\D^{[M]}_{E} = (\D^{\prime [M]}_{E^{\prime}})^{\prime}_{b}$ as sets by Propositions \ref{p:dualConvSp} and \ref{p:dualDEinOC}, i.e. $\D^{[M]}_{E}$ is semi-reflexive. As $\D^{[M]}_{E}$ is a Fr\'{e}chet space (an $(LB)$-space), it is barreled, so we may conclude that it is reflexive.
	\end{proof}
	
We now go further into the structure of $\D^{\prime [M]}_{E^{\prime}}$. For our proof to work in the Roumieu case though, we will have to establish the so-called projective description of $\D^{\{M\}}_{E}$. To this purpose, we introduce the following set
	\[ \mathfrak{R} = \{ r = (r_{j})_{j \in \N} \in \R_{+}^{\N} : r_{j} \rightarrow \infty \text{ and } r_{j} \leq r_{j + 1}, \forall j \in \N \} . \]
Then, for any $r \in \mathfrak{R}$ and $M$, we define the weight sequence $M_{r} = (M_{p} \prod_{j = 0}^{p} r_{j})_{p \in \N}$. We find the next result. 

	\begin{lemma}
		\label{l:DEProjDescr}
		We have the following isomorphism as locally convex spaces
			\[ \D^{\{M\}}_{E} \cong \varprojlim_{r \in \mathfrak{R}} D^{M_{r}, 1}_{E}  .\]
	\end{lemma}
	
	\begin{proof}
		Write $\widetilde{\D}^{\{M\}}_{E} = \varprojlim_{r \in \mathfrak{R}} D^{M_{r}, 1}_{E}$, then it follows from \cite[Lemma 3.4]{K-Ultradistributions3} that $\D^{\{M\}}_{E}$ and $\widetilde{\D}^{\{M\}}_{E}$ coincide as sets. It is clear that $\D^{\{M\}}_{E}$ is continuously contained in $\widetilde{\D}^{\{M\}}_{E}$, so we now show the converse is also true. Take any $p \in \csn(\D^{\{M\}}_{E})$ and let $B \subset \D^{\prime \{M\}}_{E^{\prime}}$ be the polar of the closed unit ball of $p$. By the bipolar theorem we have that
			\[ p(\varphi) = \sup_{f \in B} |\ev{f}{\varphi}| , \qquad \varphi \in \D^{\{M\}}_{E} . \]
		The set $B$ is strongly bounded in $\D^{\prime \{M\}}_{E^{\prime}}$, hence also in $\OC^{\prime}(\mathcal{S}^{\{M\}}_{\{A\}}, E^{\prime})$ by Proposition \ref{p:dualDEinOC}. Take $\psi \in \mathcal{S}^{(M)}_{(A)}$ as in Lemma \ref{l:AdequateWindow}, then by Proposition \ref{p:STFTConvSp} we have
			\[ \sup_{f \in B} \sup_{\xi \in \R^{d}} \norm{V_{\psi} f(\xi)}_{E^{\prime}} e^{- \omega_{M}(q \xi)} < \infty , \qquad \forall q > 0 . \]
		Using \cite[Lemma 4.5]{D-V-V-OptEmbUltraDistDiffAlg} there is some $r \in \mathfrak{R}$ and $C > 0$ such that
			\[ \norm{V_{\psi} f(\xi)}_{E^{\prime}} \leq C e^{\omega_{M_{r}}(\xi)} , \qquad \forall f \in B, \xi \in \R^{d} . \]
		In view of \cite[Lemma 2.3]{P-LaplTransUltraDist} we may assume that $M_{r}$ satisfies $(M.2)'$. Choose $s = (r_{j} / H^{d + 1}) \in \mathfrak{R}$, then for any $\varphi \in \widetilde{\D}^{\{M\}}_{E}$, by the same calculation as in the proof of Proposition \ref{p:STFTDE}, we have that $\norm{V_{\overline{\psi}} \varphi(-\xi)}_{E} \leq C' \norm{\varphi}_{\D^{M_{s}, 1}_{E}} e^{- \omega_{M_{s}}(\xi)}$ for any $\xi \in \R^{d}$ and some $C' > 0$. Then, for arbitrary $f \in B$ and $\varphi \in  \widetilde{\D}^{\{M\}}_{E}$, by \eqref{eq:M2'},
			\[ \int_{\R^{d}} \ev{V_{\psi} f(\xi)}{V_{\overline{\psi}} \varphi(-\xi)} d\xi \leq C C' \norm{\varphi}_{\D^{M_{s}, 1}_{E}} \int_{\R^{d}} \frac{e^{\omega_{M_{r}}(\xi)}}{e^{\omega_{M_{r}}(H^{d+1} \xi)}} d\xi < \infty . \]
		As $\mathcal{S}^{\{M\}}_{\{A\}}$ is dense in $\widetilde{\D}^{\{M\}}_{E}$ (it is dense in $\D^{\{M\}}_{E}$ by Corollary \ref{c:GSDenseDE}) we have by \eqref{eq:STFTDesingularizationTempUltra} that $\ev{f}{\varphi} = \int_{\R^{d}} \ev{V_{\psi} f(\xi)}{V_{\overline{\psi}} \varphi(-\xi)} d\xi$ for any $f \in B$ and $\varphi \in \widetilde{\D}^{\{M\}}_{E}$. Consequently, if we put $C_{p} = C C' \int_{\R^{d}} \frac{e^{\omega_{M_{r}}(\xi)}}{e^{\omega_{M_{r}}(H^{d+1} \xi)}} d\xi$, it follows that
			\[ p(\varphi) \leq C_{p} \norm{\varphi}_{\D^{M_{s}, 1}_{E}} , \qquad \forall \varphi \in \D^{\{M\}}_{E} , \]
		so that in particular $p \in \csn(\widetilde{\D}^{\{M\}}_{E})$.
	\end{proof}

We now obtain the ensuing structural theorem for $\OC^{\prime}(\mathcal{S}^{[M]}_{[A]}, E^{\prime})$. 
	
	\begin{theorem}
		\label{t:StructThmDualDE}
		For any $f \in \mathcal{S}^{\prime [M]}_{[A]}$, the following statements are equivalent:
			\begin{itemize}
				\item[$(i)$] $f * \varphi \in E^{\prime}$ for any $\varphi \in \mathcal{S}^{[M]}_{[A]}$;
				\item[$(ii)$] $f \in \D^{\prime [M]}_{E^{\prime}}$;
				\item[$(iii)$] there exists $(f_{\alpha})_{\alpha \in \N^{d}} \in \Lambda_{[M]}(E^{\prime})$ such that $f = \sum_{\alpha \in \N^{d}} f^{(\alpha)}_{\alpha}$.
			\end{itemize}
	\end{theorem}
	
	\begin{proof}
		$(i) \Leftrightarrow (ii)$. Follows by Proposition \ref{p:dualDEinOC}.
	
		$(ii) \Rightarrow (iii)$. Take any $f \in \D^{\prime [M]}_{E^{\prime}}$, then by the Hahn-Banach theorem we may extend $f$ to an element of $(\D^{N, \ell}_{E})^{\prime}$ for $N = M$ and some $\ell > 0$ ($N = M_{r}$ for some $r \in \mathfrak{R}$ and $\ell = 1$ in view of Lemma \ref{l:DEProjDescr}). The mapping $j : \D^{N, \ell}_{E} \rightarrow \Lambda_{1/N, 1/\ell}(E)$ given by $j(\varphi) = ((-1)^{|\alpha|} \varphi^{(\alpha)})_{\alpha \in \N^{d}}$ is an isometry, so that $\ev{f}{j(\varphi)} = \ev{f}{\varphi}$ defines a continuous linear map on $j(\D^{N, \ell}_{E})$. Another application of the Hahn-Banach theorem allows us to extend $f$ to an element of $\Lambda_{[M]}(E^{\prime})$ (where in the Roumieu case we used \cite[Lemma 3.4]{K-Ultradistributions3}), after which the representation as in $(iii)$ follows by transposition. 
	
		$(iii) \Rightarrow (ii)$. Follows directly from the structure of $f$.
	\end{proof}
	
The space $\mathcal{S}^{[M]}_{[A]}$ is in general not dense in $\OC^{\prime}(\mathcal{S}^{[M]}_{[A]}, E^{\prime})$, so one may consider its closure separately. Interestingly enough, this turns out to be a convolutor space as well. We introduce the Banach space $\dot{E}^{\prime}$ as the closure of $\mathcal{S}^{[M]}_{[A]}$ in $E^{\prime}$. Then clearly $\dot{E}^{\prime}$ is TMIB of class $[M] - [A]$ satisfying \eqref{eq:GrowthCondModulation}. Of course, if $E^{\prime}$ is a TMIB, then $E^{\prime} = \dot{E}^{\prime}$. We now find the following characterization.

	\begin{theorem}
		\label{t:closureGSindualDE}
		Let $f \in \mathcal{S}^{\prime [M]}_{[A]}$. Then, the following statements are equivalent:
			\begin{itemize}
				\item[$(i)$] $f * \varphi \in \dot{E}^{\prime}$ for any $\varphi \in \mathcal{S}^{[M]}_{[A]}$;
				\item[$(ii)$] $f$ is contained in the closure of $\mathcal{S}^{[M]}_{[A]}$ in $\OC^{\prime}(\mathcal{S}^{[M]}_{[A]}, E^{\prime})$;
				\item[$(iii)$] there exists $(f_{\alpha})_{\alpha \in \N^{d}} \in \Lambda_{[M]}(\dot{E}^{\prime})$ such that $f = \sum_{\alpha \in \N^{d}} f^{(\alpha)}_{\alpha}$.
			\end{itemize}
	\end{theorem}
	
	\begin{proof}
		$(i) \Rightarrow (ii)$. Take $\psi \in \mathcal{S}^{(M)}_{(A)}$ as in Lemma \ref{l:AdequateWindow}. By Proposition \ref{p:STFTConvSp} we have that $V_{\psi} f \in C_{[M], d}(\dot{E}^{\prime})$. By Lemma \ref{l:DensityCMSp1} and Lemma \ref{l:DensityCMSp2}, we see that $C_{[M]}(\mathcal{S}^{[M]}_{[A]})$ is dense in $C_{[M], d}(\dot{E}^{\prime})$, hence there is a net $(\Phi_{\tau})_{\tau}$ in $C_{[M]}(\mathcal{S}^{[M]}_{[A]})$ such that $\lim_{\tau} \Phi_{\tau} = V_{\psi} f$. Note that $C_{[M], d}(\dot{E}^{\prime})$ is a topological subspace of $C_{[M], d}(E^{\prime})$, so the same limit holds there. Then, in view of Proposition \ref{p:STFTGS}, Proposition \ref{p:STFTConvSp} and \eqref{eq:STFTReconstructTempUltra}, it follows that $\lim_{\tau} V^{*}_{\psi} \Phi_{\tau} = f$ in $\OC^{\prime}(\mathcal{S}^{[M]}_{[A]}, E^{\prime})$ where $(V^{*}_{\psi} \Phi_{\tau})_{\tau}$ is a net in $\mathcal{S}^{[M]}_{[A]}$.
		
		$(ii) \Rightarrow (i)$. Suppose that $f \in \OC^{\prime}(\mathcal{S}^{[M]}_{[A]}, E^{\prime})$ is the limit of some net $(\psi_{\tau})_{\tau}$ in $\mathcal{S}^{[M]}_{[A]}$. In particular $f \in \D^{\prime [M]}_{E^{\prime}}$ by Proposition \ref{p:dualDEinOC}.  For any $\varphi \in \mathcal{S}^{[M]}_{[A]}$, in view of Lemma \ref{l:ConvolutionEtoDE}, $\{ e * \check{\varphi} : e \in E, \norm{e}_{E} \leq 1 \}$ is a bounded subset of $\D^{[M]}_{E}$. Hence, it follows that
			\[ \lim_{\tau} \norm{f * \varphi - \psi_{\tau} * \varphi}_{E^{\prime}} = \lim_{\tau} \sup_{\norm{e}_{E} \leq 1} |\ev{f - \psi_{\tau}}{e * \check{\varphi}}| = 0 . \]
		As $(\psi_{\tau} * \varphi)_{\tau}$ is a net in $\mathcal{S}^{[M]}_{[A]}$ by Lemma \ref{l:ConvolutionGS}, we see that $f * \varphi \in \dot{E}^{\prime}$.
	
		$(i) \Leftrightarrow (iii)$. This follows directly from Theorem \ref{t:StructThm}.
	\end{proof}

We end this section by considering some concrete examples. Take a measurable function $\omega : \R^{d} \rightarrow (0, \infty)$ such that $\omega$ and $\omega^{-1}$ are locally bounded. For a given weight sequence $N$, we say that $\omega$ is \emph{$(N)$-admissible} (\emph{$\{N\}$-admissible}) if
	\[ \exists q > 0 ~ (\forall q > 0) ~ \exists C > 0 ~ \forall x, t \in \R^{d} : \omega(x + t) \leq C \omega(x) e^{\omega_{N}(q t)} . \]
For example, by \eqref{eq:TMIBnormcond} we have that for some $q > 0$ (for every $q > 0$) and for all $x, t \in \R^{d}$:
	\begin{multline*} 
		\omega_{E}(x + t) = \|T_{x + t}\|_{L(E)} = \|T_{x} T_{t}\|_{L(E)} \\ \leq \|T_{x}\|_{L(E)} \|T_{t}\|_{L(E)} = \omega_{E}(x) \omega_{E}(t) \leq C_{E, q} \omega_{E}(x) e^{\omega_{A}(q t)} .
	\end{multline*}
Hence, $\omega_{E}$ is $[A]$-admissible. Similarly, $\nu_{E}$ is $[M]$-admissible. Also remark that if $\omega$ is $[N]$-admissible, then in particular for some $q > 0$ (for every $q > 0$) there is a $C > 0$ such that for all $x \in \R^{d}$:
	\[ \frac{\omega(0)}{C} e^{-\omega_{N}(qx)} \leq \omega(x) \leq C \omega(0) e^{\omega_{N}(qx)} . \] 
Now, for any $p \in [1, \infty)$ we define $L^{p}_{\omega} = L^{p}_{\omega}(\R^{d})$ as the Banach space consisting of all measurable functions $\varphi$ on $\R^{d}$ such that
	\[ \norm{\varphi}_{L^{p}_{\omega}} := \left( \int_{\R^{d}} |\varphi(x) \omega(x)|^{p} dx \right)^{\frac{1}{p}} < \infty , \]
while $L^{\infty}_{\omega}$ is defined as the Banach space of all measurable functions $\varphi$ on $\R^{d}$ such that
	\[ \norm{\varphi}_{L^{\infty}_{\omega}} := \text{ess\,sup}_{x \in \R^{d}} |\varphi(x)| / \omega(x) < \infty . \]
Using \cite[Theorem 1.2]{D-N-V-NuclGSKerThm}, one can easily show that, for $p \in [1, \infty)$, if $\omega$ is $[A]$-admissible then $L^{p}_{\omega}$ is a TMIB of class $[M] - [A]$ satisfying \eqref{eq:GrowthCondModulation}. Also, for $p \in (1, \infty)$, we have that $(L^{p}_{\omega})^{\prime} = L^{q}_{\omega^{-1}}$ where $q \in (1, \infty)$ is such that $\frac{1}{p} + \frac{1}{q} = 1$, while the dual of $L^{1}_{\omega}$ is $L^{\infty}_{\omega}$. By Theorem \ref{t:StructThm} and Theorem \ref{t:StructThmDualDE}, we get the following structural theorem.	

	\begin{theorem}
		Let $p \in [1, \infty]$ and $\omega$ be an $[A]$-admissible weight function. Then, for any $f \in \mathcal{S}^{\prime [M]}_{[A]}$, 
			\[ f * \varphi \in L^{p}_{\omega} , \qquad \forall \varphi \in \mathcal{S}^{[M]}_{[A]} , \]
		if and only if there exists $(f_{\alpha})_{\alpha \in \N^{d}} \in \Lambda_{[M]}(L^{p}_{\omega})$ such that
			\[ f = \sum_{\alpha \in \N^{d}} f^{(\alpha)}_{\alpha} . \]
	\end{theorem}
	
We now further consider the space $\D^{\prime [M]}_{L^{\infty}_{\omega}}$, more commonly denoted by $\mathcal{B}^{\prime [M]}_{\omega}$. The DTMIB $L^{\infty}_{\omega}$ is not a TMIB, and $\mathcal{S}^{[M]}_{[A]}$ is not dense in $\mathcal{B}^{\prime [M]}_{\omega}$. Hence one may consider its closure, denoted by $\dot{\mathcal{B}}^{\prime [M]}_{\omega}$. Now, one sees that $\dot{L}^{\infty}_{\omega} = C_{0, \omega}$, where the latter is the space of all continuous functions $\varphi$ on $\R^{d}$ such that $\lim_{|x| \rightarrow \infty} \varphi(x) / \omega(x) = 0$. Consequently, $\dot{\mathcal{B}}^{\prime [M]}_{\omega}$ is exactly the space $\OC^{\prime}(\mathcal{S}^{[M]}_{[A]}, C_{0, \omega})$, and for this reason it is also known as the space of ultradistributions vanishing at infinity w.r.t. $\omega$. By Theorem \ref{t:closureGSindualDE} we find the ensuing structural result, first obtained in \cite{D-N-V-SpUltraDistVanishInf}.

	\begin{theorem}[{\cite[Theorem 1]{D-N-V-SpUltraDistVanishInf}}]
		Let $\omega$ be an $[A]$-admissible weight function. Then, for any $f \in \mathcal{S}^{\prime [M]}_{[A]}$, the following statements are equivalent:
			\begin{itemize}
				\item[$(i)$] $f * \varphi \in C_{0, \omega}$ for any $\varphi \in \mathcal{S}^{[M]}_{[A]}$;
				\item[$(ii)$] $f \in \dot{\mathcal{B}}^{\prime [M]}_{\omega}$;
				\item[$(iii)$] there exists $(f_{\alpha})_{\alpha \in \N^{d}} \in \Lambda_{[M]}(C_{0, \omega})$ such that $f = \sum_{\alpha \in \N^{d}} f^{(\alpha)}_{\alpha}$.
			\end{itemize}
	\end{theorem}
	
We may also look at the frequency side. Let $\omega$ be $[M]$-admissible and we assume one of the following two conditions:
	\begin{itemize}
		\item[($\mathcal{F}$.I)] $\exists C, C', k >  0 : C (1 + |\xi|)^{-k} \leq \omega(\xi) \leq C' (1 + |\xi|)^{k}$;
		\item[($\mathcal{F}$.II)] $M$ satisfies $(M.2)$.
	\end{itemize}
For $p \in [1, \infty]$ we then define the Banach space $\mathcal{F}L^{p}_{\omega}$ as all those $f$ such that its Fourier transform $\widehat{f} \in L^{p}_{\omega}$, with the norm $\|f\|_{\mathcal{F}L^{p}_{\omega}} = \|\widehat{f}\|_{L^{p}_{\omega}}$, where $f$ is an element of $\mathcal{S}^{\prime}$ if ($\mathcal{F}$.I) holds, or of $\mathcal{S}^{\prime [M]}_{[A]}$ if ($\mathcal{F}$.II) holds (note that if $(M.2)$ holds the Fourier transform is an isomorphism between $\mathcal{S}^{[M]}_{[A]}$ and $\mathcal{S}^{[A]}_{[M]}$, cfr. \cite{C-C-K-CharGelfandShilovFourier}). Then, for $p \in [1, \infty)$, $\mathcal{F}L^{p}_{\omega}$ is a TMIB of class $[M] - [A]$ satisfying \eqref{eq:GrowthCondModulation}. Moreover, $\mathcal{F}L^{\infty}_{\omega}$ is the dual of $\mathcal{F}L^{1}_{\omega}$ so that it is a DTMIB. Once again by Theorem \ref{t:StructThm} and Theorem \ref{t:StructThmDualDE} we find the following result.

	\begin{theorem}
		Let $p \in [1, \infty]$ and $\omega$ be an $[M]$-admissible weight function such that either ($\mathcal{F}$.I) or ($\mathcal{F}$.II) holds. Then, for any $f \in \mathcal{S}^{\prime [M]}_{[A]}$, 
			\[ f * \varphi \in \mathcal{F}L^{p}_{\omega} , \qquad \forall \varphi \in \mathcal{S}^{[M]}_{[A]} , \]
		if and only if there exists $(f_{\alpha})_{\alpha \in \N^{d}} \in \Lambda_{[M]}(\mathcal{F}L^{p}_{\omega})$ such that
			\[ f = \sum_{\alpha \in \N^{d}} f^{(\alpha)}_{\alpha} . \]
	\end{theorem}

Analogously, by Theorem \ref{t:closureGSindualDE}, we find the following result for $\mathcal{F}C_{0, \omega}$, the Banach space of all $f \in \mathcal{F}L^{\infty}_{\omega}$ such that $\widehat{f} \in C_{0, \omega}$.

	\begin{theorem}
		Let $\omega$ be an $[M]$-admissible weight function such that either ($\mathcal{F}$.I) or ($\mathcal{F}$.II) holds. Then, for any $f \in \mathcal{S}^{\prime [M]}_{[A]}$, the following statements are equivalent:
			\begin{itemize}
				\item[$(i)$] $f * \varphi \in \mathcal{F}C_{0, \omega}$ for any $\varphi \in \mathcal{S}^{[M]}_{[A]}$;
				\item[$(ii)$] $f$ is contained in the closure of $\mathcal{S}^{[M]}_{[A]}$ in $\OC^{\prime}(\mathcal{S}^{[M]}_{[A]}, \mathcal{F}L^{\infty}_{\omega})$;
				\item[$(iii)$] there exists $(f_{\alpha})_{\alpha \in \N^{d}} \in \Lambda_{[M]}(\mathcal{F}C_{0, \omega})$ such that $f = \sum_{\alpha \in \N^{d}} f^{(\alpha)}_{\alpha}$.
			\end{itemize}
	\end{theorem}

\section{Extensions of convolution}
\label{sec:ExtensionConv}

In this final section we study several extensions of convolution arising naturally from the structure found in Theorem \ref{t:StructThm}. Throughout this section we assume $M$ and $A$ are weight sequences satisfying $(M.1)$ and $(M.2)'$, where $M$ additionally satisfies $(M.2)$, and $E$ is a TMIB of class $[M] - [A]$ (which trivially satisfies \eqref{eq:GrowthCondModulation}, see Remark \ref{r:GrowthCondModulationSuffCond}). 

We will often work with spaces associated to $L^{1}_{\omega_{E}}$. Note that $\omega_{E}$ is $[A]$-admissible, so that $L^{1}_{\omega_{E}}$ is a TMIB of class $[M] - [A]$ which satisfies \eqref{eq:GrowthCondModulation}, see Section \ref{sec:DTMIB}. In particular, we have that $\mathcal{S}^{[M]}_{[A]}$ is dense in $\D^{[M]}_{L^{1}_{\omega_{E}}}$ by Corollary \ref{c:GSDenseDE}, while Theorem \ref{t:StructThm} is valid for the elements in $\OC^{\prime}(\mathcal{S}^{[M]}_{[A]}, L^{1}_{\omega_{E}})$.

We first look at $\OC^{\prime}(\mathcal{S}^{[M]}_{[A]}, E)$ itself and show that its elements may convolve with those of $\D^{[M]}_{L^{1}_{\omega_{E}}}$ and moreover those convolutions are all contained in $\D^{[M]}_{E}$. To do this, we first observe the following continuous mapping.

	\begin{lemma}
		\label{l:ExtendConvStructSmoothCase}
		The mapping
			\[ \Lambda_{[M]}(E) \times \D^{[M]}_{L^{1}_{\omega_{E}}} \rightarrow \D^{[M]}_{E} : \quad ((f_{\alpha})_{\alpha \in \N^{d}}, \psi) \mapsto \sum_{\alpha \in \N^{d}} f_{\alpha} * \psi^{(\alpha)} , \]
		is well-defined and continuous.
	\end{lemma}
	
	\begin{proof}
		For any $\ell > 0$, suppose $(f_{\alpha})_{\alpha \in \N^{d}} \in \Lambda_{M, 2\ell}(E)$ and $\psi \in \D^{M, \ell / H}_{L^{1}_{\omega_{E}}}$, then, by \eqref{eq:convolutionineqE}, we have for $\gamma \in \N^{d}$,
			\begin{align*}
				\| \partial^{\gamma} \sum_{\alpha \in \N^{d}} f_{\alpha} * \psi^{(\alpha)} \|_{E} &\leq \sum_{\alpha \in \N^{d}} \| f_{\alpha} * \psi^{(\alpha + \gamma)} \|_{E} \leq \sum_{\alpha \in \N^{d}} \| f_{\alpha} \|_{E} \| \psi^{(\alpha + \gamma)} \|_{L^{1}_{\omega_{E}}} \\
				&\leq \|f\|_{\Lambda_{M, 2 \ell}(E)} \|\psi\|_{\D^{M, \ell / H}_{L^{1}_{\omega_{E}}}} \sum_{\alpha \in \N^{d}} \frac{(\ell / H)^{|\alpha| + |\gamma|} M_{\alpha + \gamma}}{(2 \ell)^{|\alpha|} M_{\alpha}} \\
				&\leq 2^{d} C_{0} \ell^{|\gamma|} M_{\gamma} \|f\|_{\Lambda_{M, 2 \ell}(E)} \|\psi\|_{\D^{M, \ell / H}_{L^{1}_{\omega_{E}}}} ,
			\end{align*}
		so that $\sum_{\alpha \in \N^{d}} f_{\alpha} * \psi^{(\alpha)} \in \D^{M, \ell}_{E}$ and the continuity now follows directly. 
	\end{proof}
	
Take any $f \in \OC^{\prime}(\mathcal{S}^{[M]}_{[A]}, E)$ and let $f = \sum_{\alpha \in \N^{d}} f^{(\alpha)}_{\alpha}$ for some $(f_{\alpha})_{\alpha \in \N^{d}} \in \Lambda_{[M]}(E)$. Using Lemma \ref{l:ExtendConvStructSmoothCase}, we may define the convolution of $f$ with any $\psi \in \D^{[M]}_{L^{1}_{\omega_{E}}}$ as
	\[ f * \psi := \sum_{\alpha \in \N^{d}} f_{\alpha} * \psi^{(\alpha)} \in \D^{[M]}_{E} . \]
To see that this mapping is well-defined, one has to verify whether it is invariant under the kernel of the structure of $\OC^{\prime}(\mathcal{S}^{[M]}_{[A]}, E)$. If this is the case, we also note that this definition uniquely extends the convolution \eqref{eq:GSConv} of Gelfand-Shilov spaces due to the density of $\mathcal{S}^{[M]}_{[A]}$ in $\OC^{\prime}(\mathcal{S}^{[M]}_{[A]}, E)$ (Corollary \ref{c:denseinclusions}) and $\D^{[M]}_{L^{1}_{\omega_{E}}}$ (Corollary \ref{c:GSDenseDE}). Thus, we find our first extension of convolution.
	
	\begin{theorem}
		\label{t:ExtensionConvSmoothCase}
		The convolution mapping
			\[ * : \OC^{\prime}(\mathcal{S}^{[M]}_{[A]}, E) \times \D^{[M]}_{L^{1}_{\omega_{E}}} \rightarrow \D^{[M]}_{E} \]
		is well-defined and continuous. 
	\end{theorem}
	
	\begin{proof}
		In view of Lemma \ref{l:ExtendConvStructSmoothCase}, as $\OC^{\prime}(\mathcal{S}^{[M]}_{[A]}, E) \cong \Lambda_{[M]}(E) / \ker S$, where $S$ is the map in \eqref{eq:Smap}, to show the result we must verify that for any $(f_{\alpha})_{\alpha \in \N^{d}} \in \ker S$ and $\psi \in \D^{[M]}_{L^{1}_{\omega_{E}}}$ we have $\sum_{\alpha \in \N^{d}} f_{\alpha} * \psi^{(\alpha)} = 0$. Indeed, for such a fixed $(f_{\alpha})_{\alpha \in \N^{d}}$ we have that the mapping $\psi \mapsto \sum_{\alpha \in \N^{d}} f_{\alpha} * \psi^{(\alpha)}$ is continuous. Take any $\psi, \varphi \in \mathcal{S}^{[M]}_{[A]}$, then,
			\[ \ev{\sum_{\alpha \in \N^{d}} f_{\alpha} * \psi^{(\alpha)}}{\varphi} = \ev{\sum_{\alpha \in \N^{d}} f^{(\alpha)}_{\alpha}}{\varphi * \check{\psi}} = 0 . \]
		As $\mathcal{S}^{[M]}_{[A]}$ is dense in $\D^{[M]}_{L^{1}_{\omega_{E}}}$ by Corollary \ref{c:GSDenseDE}, it follows that $\sum_{\alpha \in \N^{d}} f_{\alpha} * \psi^{(\alpha)} = 0$ for all $\psi \in \D^{[M]}_{L^{1}_{\omega_{E}}}$. 
	\end{proof}
	
For any $f \in \OC^{\prime}(\mathcal{S}^{[M]}_{[A]}, E^{\prime})$ and $g \in \OC^{\prime}(\mathcal{S}^{[M]}_{[A]}, E)$, Theorem \ref{t:ExtensionConvSmoothCase} and Proposition \ref{p:dualDEinOC} show that we may define the convolution $f * g$ as follows,
	\[ \ev{f * g}{\varphi} := \ev{f}{\check{g} * \varphi} , \qquad \varphi \in \D^{[M]}_{L^{1}_{\omega_{E}}} . \]
Hence $f * g \in \mathcal{B}^{\prime [M]}_{\omega_{E}}$, and in the special case where $g \in \mathcal{S}^{[M]}_{[A]}$ it coincides with the convolution as in \eqref{eq:ConvTempUltraDist}. In summary, we find the ensuing second extension of convolution.
	
	\begin{theorem}
		\label{t:ExtensionConvSmoothCaseTransp}
		The convolution mapping
			\[ * : \OC^{\prime}(\mathcal{S}^{[M]}_{[A]}, E^{\prime}) \times \OC^{\prime}(\mathcal{S}^{[M]}_{[A]}, E) \rightarrow \mathcal{B}^{\prime [M]}_{\omega_{E}} \]
		is well-defined and hypocontinuous. 
	\end{theorem}
	
For our third and final extension of convolution, we consider the following natural extension of the convolution \eqref{eq:ConvWTMIB} to the corresponding sequence spaces.

	\begin{lemma}
		\label{l:ExtendConvStructUltraDistCase}
		The mapping 
			\[ \Lambda_{[M]}(E) \times \Lambda_{[M]}(L^{1}_{\omega_{E}}) \rightarrow \Lambda_{[M]}(E) : \quad ((f_{\alpha})_{\alpha \in \N^{d}}, (g_{\beta})_{\beta \in \N^{d}}) \mapsto (\sum_{\alpha + \beta = \gamma} f_{\alpha} * g_{\beta})_{\gamma \in \N^{d}} \]
		is well-defined and continuous.
	\end{lemma}
	
	\begin{proof}
		For any $\ell > 0$, let $(f_{\alpha})_{\alpha \in \N^{d}} \in \Lambda_{M,  2 \ell H}(E)$ and $(g_{\beta})_{\beta \in \N^{d}} \in \Lambda_{M, 2 \ell H}(L^{1}_{\omega_{E}})$, then, by \eqref{eq:convolutionineqE}, we have
			\begin{align*}
				\| \sum_{\alpha + \beta = \gamma} f_{\alpha} * g_{\beta} \|_{E} &\leq \sum_{\alpha + \beta = \gamma} \|f_{\alpha}\|_{E} \|g_{\beta}\|_{L^{1}_{\omega_{E}}} \\
				&\leq \sum_{\alpha + \beta = \gamma} \frac{\|(f_{\alpha})_{\alpha \in \N^{d}}\|_{\Lambda_{M, 2 \ell H}(E)}}{(2 \ell H)^{|\alpha|} M_{\alpha}} \frac{\|(g_{\beta})_{\beta \in \N^{d}}\|_{\Lambda_{M, 2 \ell H}(L^{1}_{\omega_{E}})}}{(2 \ell H)^{|\beta|} M_{\beta}} \\
				&\leq C_{0} \|(f_{\alpha})_{\alpha \in \N^{d}}\|_{\Lambda_{M, 2 \ell H}(E)} \|(g_{\beta})_{\beta \in \N^{d}}\|_{\Lambda_{M, 2 \ell H}(L^{1}_{\omega_{E}})} / (\ell^{|\gamma|} M_{\gamma}) ,
			\end{align*}
		hence $(\sum_{\alpha + \beta = \gamma} f_{\alpha} * g_{\beta})_{\gamma \in \N^{d}} \in \Lambda_{M, \ell}(E)$ and the continuity follows. 
	\end{proof}
	
For $f \in \OC^{\prime}(\mathcal{S}^{[M]}_{[A]}, E)$ with structure $f = \sum_{\alpha \in \N^{d}} f^{(\alpha)}_{\alpha}$, $(f_{\alpha})_{\alpha \in \N^{d}} \in \Lambda_{[M]}(E)$, and $g \in \OC^{\prime}(\mathcal{S}^{[M]}_{[A]}, L^{1}_{\omega_{E}})$ with structure $g = \sum_{\beta \in \N^{d}} g^{(\beta)}_{\beta}$, $(g_{\beta})_{\beta \in \N^{d}} \in \Lambda_{[M]}(L^{1}_{\omega_{E}})$, we can then apply Lemma \ref{l:ExtendConvStructUltraDistCase} and Theorem \ref{t:StructThm} to define the convolution $f * g$ as
	\[ f * g := \sum_{\gamma \in \N^{d}} \partial^{\gamma} \left( \sum_{\alpha + \beta = \gamma} f_{\alpha} * g_{\beta} \right) \in \OC^{\prime}(\mathcal{S}^{[M]}_{[A]}, E) . \]
Again, in order for this definition to be well-defined we have to verify that it is invariant under the kernel of the structure. Then, this definition will be the unique extension of the convolution $* : E \times L^{1}_{\omega_{E}} \rightarrow E$, so also of \eqref{eq:GSConv}. Our last extension of convolution may thus be written as follows.

	\begin{theorem}
		\label{t:ExtendConvUltraDistCase}
		The convolution mapping
			\[ * : \OC^{\prime}(\mathcal{S}^{[M]}_{[A]}, E) \times \OC^{\prime}(\mathcal{S}^{[M]}_{[A]}, L^{1}_{\omega_{E}}) \rightarrow \OC^{\prime}(\mathcal{S}^{[M]}_{[A]}, E) \]
		is well-defined and continuous. 
	\end{theorem}
	
	\begin{proof}
		Suppose that $(f_{\alpha})_{\alpha \in \N^{d}}$ is in the kernel of $S : \Lambda_{[M]}(E) \rightarrow \OC^{\prime}(\mathcal{S}^{[M]}_{[A]}, E)$. The function $A : (g_{\beta})_{\beta \in \N^{d}} \mapsto \sum_{\gamma \in \N^{d}} \partial^{\gamma} (\sum_{\alpha + \beta = \gamma} f_{\alpha} * g_{\beta})$ is a continuous linear map $\Lambda_{[M]}(L^{1}_{\omega_{E}}) \rightarrow \OC^{\prime}(\mathcal{S}^{[M]}_{[A]}, E)$. Take any $(\psi_{\beta})_{\beta \in \N^{d}} \in \Lambda_{[M]}(E)$ such that $\psi_{\beta} \in \mathcal{S}^{[M]}_{[A]}$ for every $\beta \in \N^{d}$ while $\psi_{\beta} \neq 0$ for a finite amount of $\beta$. For arbitrary $\varphi \in \mathcal{S}^{[M]}_{[A]}$ we define $\psi = \sum_{\beta \in \N^{d}} \varphi * \check{\psi}_{\beta}^{(\beta)}$, then $\psi \in \mathcal{S}^{[M]}_{[A]}$. Now,
			\begin{multline*} 
				\ev{\sum_{\gamma \in \N^{d}} \partial^{\gamma} \sum_{\alpha + \beta = \gamma} f_{\alpha} * \psi_{\beta}}{\varphi}
				= \sum_{\gamma \in \N^{d}} \sum_{\alpha + \beta = \gamma} (-1)^{|\alpha| + |\beta|} \ev{f_{\alpha}}{\varphi^{(\alpha + \beta)} * \check{\psi}_{\beta}} \\
				= \sum_{\gamma \in \N^{d}} \sum_{\alpha + \beta = \gamma} \ev{f^{(\alpha)}_{\alpha}}{\varphi * \check{\psi}_{\beta}^{(\beta)}} 
				= \sum_{\alpha \in \N^{d}} \ev{f_{\alpha}^{(\alpha)}}{\psi} = 0 .
			\end{multline*}
		As the subspace of all such $(\psi_{\beta})_{\beta \in \N^{d}}$ is dense in $\Lambda_{[M]}(L^{1}_{\omega_{E}})$, the map $A$ is identically zero. Similarly, one shows that for any $(g_{\beta})_{\beta \in \N^{d}}$ in the kernel of $S : \Lambda_{[M]}(L^{1}_{\omega_{E}}) \rightarrow \OC^{\prime}(\mathcal{S}^{[M]}_{[A]}, E)$, the linear map $(f_{\alpha})_{\alpha \in \N^{d}} \mapsto \sum_{\gamma \in \N^{d}} \partial^{\gamma} (\sum_{\alpha + \beta = \gamma} f_{\alpha} * g_{\beta})$ from $\Lambda_{[M]}(E)$ into $\OC^{\prime}(\mathcal{S}^{[M]}_{[A]}, E)$ is identically zero. This completes the proof.
	\end{proof}

\end{document}